\documentclass[12pt]{amsart}

\usepackage{amsmath}
\usepackage{amssymb}
\usepackage{array}
\usepackage{enumitem}
\usepackage{url}
\usepackage{verbatim} 

\input cyracc.def
\newfam\cyrfam

\theoremstyle{plain}
\newtheorem{theorem}{Theorem}[section]
\newtheorem{lemma}[theorem]{Lemma}
\newtheorem{proposition}[theorem]{Proposition}
\newtheorem{corollary}[theorem]{Corollary}

\theoremstyle{definition}
\newtheorem*{definition}{Definition}
\theoremstyle{remark}
\newtheorem*{remark}{Remark}

\newtheorem*{notation}{Notation}

\numberwithin{equation}{section}

\newcommand{\seclabel}[1]{\label{sec:#1}}   
\newcommand{\thmlabel}[1]{\label{thm:#1}}   
\newcommand{\eqnlabel}[1]{\label{eqn:#1}}   

\newcommand{\eqnref}[1]{\eqref{eqn:#1}} 

\newcommand{\bA}{\mathbb{A}}
\newcommand{\AAA}{\mathbb{A}}
\newcommand{\bB}{\mathbb{B}}
\newcommand{\B}{\mathbb{B}}
\newcommand{\BB}{\mathbb{B}}
\newcommand{\K}{\mathbb{K}}
\newcommand{\bK}{\mathbb{K}}

\newcommand{\R}{\mathbb{R}}
\newcommand{\C}{\mathbb{C}}
\newcommand{\Z}{\mathbb{Z}}
\newcommand{\PP}{\mathbb{P}}

\newcommand{\HHH}{\mathbb{H}}
\newcommand{\bH}{\mathbb{H}}

\newcommand{\cX}{\mathcal{X}}
\newcommand{\XX}{\mathcal{X}}
\newcommand{\cY}{\mathcal{Y}}
\newcommand{\YY}{\mathcal{Y}}
\newcommand{\cP}{{\mathcal P}}
\newcommand{\cG}{{\mathcal G}}
\newcommand{\GG}{{\mathcal G}}

\newcommand{\sP}{\mathfrak{sp}}
\newcommand{\oo}{\mathfrak{o}}
\newcommand{\so}{\mathfrak{so}}
\newcommand{\gL}{\mathfrak{gl}}
\newcommand{\uu}{\mathfrak{u}}

\newcommand{\Gl}{\mathrm{GL}}
\newcommand{\UU}{\mathrm{U}}
\newcommand{\OO}{\mathrm{O}}

\newcommand{\Sp}{\mathrm{Sp}}
\newcommand{\Hom}{\mathrm{Hom}}
\newcommand{\Herm}{\mathrm{Herm}}
\newcommand{\Sym}{\mathrm{Sym}}
\newcommand{\Aherm}{\mathrm{Aherm}}
\newcommand{\Asym}{\mathrm{Asym}}
\newcommand{\Aut}{\mathrm{Aut}}
\newcommand{\End}{\mathrm{End}}

\newcommand{\Gras}{\mathrm{Gras}}

\newcommand{\im}{\mathrm{im}}

\newcommand{\id}{\mathrm{id}}
\newcommand{\dom}{\mathrm{dom}}
\newcommand{\indef}{\mathrm{indef}}
\newcommand{\pr}{\mathrm{pr}}

\newcommand{\eps}{\varepsilon}

\newcommand{\Bigsetof}[2]{\begin{Bmatrix} #1 \,\Big|\, #2 \end{Bmatrix}}

\newcommand{\inv}{^{-1}}
\newcommand{\msk}{\medskip}
\newcommand{\ssk}{\smallskip}
\newcommand{\nin}{\noindent}

\begin{document}

\title[Associative Geometries. II]{Associative Geometries. II: Involutions,
the classical torsors, and their homotopes}

\author{Wolfgang Bertram}

\address{Institut \'{E}lie Cartan Nancy \\
Nancy-Universit\'{e}, CNRS, INRIA \\
Boulevard des Aiguillettes, B.P. 239 \\
F-54506 Vand\oe{}uvre-l\`{e}s-Nancy, France}

\email{\url{bertram@iecn.u-nancy.fr}}

\author{Michael Kinyon}

\address{Department of Mathematics \\
University of Denver \\
2360 S Gaylord St \\
Denver, Colorado 80208 USA}

\email{\url{mkinyon@math.du.edu}}



\begin{abstract}
For all classical groups (and for their analogs in infinite
dimension or over general base fields or rings)
we construct certain contractions, called \emph{homotopes}.
The construction is geometric, using as ingredient \emph{involutions
of associative geometries}. We prove that, under suitable
assumptions, the groups and their homotopes have a canonical  semigroup
completion.
\end{abstract}

\maketitle

\section*{Introduction: The classical groups revisited}
\seclabel{intro}
The purpose of this work is to explain two remarkable features
of classical groups:
\begin{enumerate}
\item
every classical group is a member of a  ``continuous'' family
interpolating between the group and its ``flat'' Lie algebra;
put differently, there is a geometric construction of ``contractions''
(in this context also called \emph{homotopes}),
\item
every classical group and all of its homotopes
 admit a canonical completion to a semigroup;
 the underlying (compact) space of all of these ``semigroup hulls"
is the same for all homotopes.
\end{enumerate}
In fact, these results hold much more generally. The key property
of classical groups is that they are closely
related to
\emph{associative algebras}: either they are (quotients of)
unit groups of such algebras, or
they are (quotients of) \emph{$*$-unitary groups}
\begin{equation}
\UU(\bA,*):=\{ u \in \bA | \, uu^* = 1 \}
\label{00}
\end{equation}
for some \emph{involutive associative algebra} $(\bA,*)$.
This way of characterizing classical groups
suggests to consider as ``classical'' also all other groups given by
these constructions, including
 infinite-dimensional groups
and groups over general base fields or rings $\K$,
obtained from general involutive associative algebras $(\bA,*)$ over
$\K$.

\ssk
On an algebraic, or ``infinitesimal'', level, features (1) and (2)
are supported by simple observations on associative algebras:
as to (1), associative algebras really are families of products
$(x, y) \mapsto xay$ (the homotopes, see below), and as to (2), it is
obvious that an associative algebra forms a semigroup and not
a group with respect to multiplication.
Our task is, then, to ``globalize'' these simple observations, and
at the same time to put them into the form of a geometric
theory: we have to free them from choices of base points (such as
$0$ and the unit $1$ in an associative algebra). Just as in classical
geometry, this means to proceed from a ``linear'' to a ``projective''
formulation, with an ``affine'' formulation as intermediate step.
For classical groups of the ``general linear type'' ($A_n$), this
has already been achieved in Part I of this work (\cite{BeKi09}).
In the present  article we look at the remaining families ($B_n$,
$C_n$, $D_n$) and their generalizations.
They correspond to associative algebras \emph{with involution},
so that the geometric theory of \emph{involutions} will be a
central topic of this work.
Let us start by describing the ``infinitesimal'' situation
(i.e., the Lie algebra level), before explaining how to
``globalize'' it.

\subsection{Homotopes of classical Lie algebras}

The concept of \emph{homotopy} is at the base of Part I of this work:
an associative algebra $\AAA$ should be seen rather as a \emph{family} of
associative algebras $(\AAA, (x,y) \mapsto xay)$, parametrized by
$a \in \bA$. This  gives rise to a family of
Lie brackets $[x,y]_a = xay - yax$ also called {\it homotopes},
interpolating between the ``usual'' Lie bracket ($a=1$) and
the trivial one ($a=0$).
In particular, taking for $\bA$ the matrix space $M(n,n;\K)$
with Lie bracket $[X,Y]_A$ for $A \in M(n,n;\K)$ we get
a Lie algebra which will be denoted by $\gL_n(A;\K)$.

\ssk
For abstract Lie algebras, there is no such construction;
however, there is a variant that can be applied to
all classical Lie algebras:
let us add an \emph{involution} $*$ (antiautomorphism of order
$2$) as a new structural feature to our associative algebra
$\bA$, and write
$$
\bA = \Herm(\AAA,*) \oplus \Aherm(\AAA,*) =
\{ x \in \bA | \, x^* =x \} \oplus \{ x \in \bA | \, x^* = -x \}
$$
for the eigenspace decomposition.
If we fix $a \in \Herm(\AAA,*)$, then
$*:\bA \to \bA$ is an antiautomorphism of the
homotopic bracket $[ \cdot,\cdot]_a$,
and therefore $(\Aherm(\AAA,*),[ \cdot,\cdot]_a)$, $a \in \Herm(\bA,*)$,
is a family of Lie algebra structures on $\Aherm(\AAA,*)$, again
called homotopes.
Remarkably, the construction works also in the other direction:
if we fix $a \in \Aherm(\AAA,*)$, then
$\bA \to \bA$, $x \mapsto x^*$
is an automorphism of the homotope bracket $[ \cdot,\cdot]_a$,
and hence $(\Herm(\AAA,*),[ \cdot,\cdot]_a)$, $a \in \Aherm(\bA,*)$,
is a family of ``homotopic'' Lie algebra structures on $\Herm(\AAA,*)$.
For instance, taking for $\bA$ the matrix algebra $M(n,n;\K)$ with
involution $X^*:=X^t$ (transposed matrix; in this case we write
$\Sym(n;\K)$ and $\Asym(n;\K)$ for the eigenspaces), we get
contractions of the \emph{orthogonal Lie algebras},
denoted by $\oo_n(A;\K):=\Asym(n;\K)$ with bracket
$[X,Y]_A$ for symmetric matrices $A$. For $A=1$,
we get the usual Lie algebra $\oo(n)$; for $A = I_{p,q}$ (diagonal
matrix of signature $(p,q)$ with $p+q=n$) we get the pseudo-orthogonal
algebras $\oo(p,q)$, but for $p+q < n$ we get a new kind of Lie algebras:
they are \emph{not} Lie algebras defined by a form since  the Lie algebra of a degenerate form
has bigger dimension than the one of a non-degenerate form,
whereas our contractions preserve dimension.
Likewise, for skew-symmetric $A$, we get homotopes of
``symplectic type''
$\sP_{n/2}(A;\K):=\Sym(n;\K)$ with bracket $[X,Y]_A$.
If $n$ is even and $A$ invertible, then this algebra is isomorphic
to the usual symplectic algebra $\sP(n,\K)$, and if $A$ is not invertible,
we get ``degenerate'' homotopes; as in the
orthogonal case, these algebras are not Lie algebras defined by a
degenerate skewsymmetric form. If $n$ is odd, then the family
contains only ``degenerate members'', which we call
\emph{half-symplectic}.

\ssk
Summing up, looking at homotope Lie brackets on $\Aherm(\bA,*)$
not only serves to imbed the usual Lie bracket into a family,
but also to restore a remarkable formal duality between
$\Aherm(\bA)$ and $\Herm(\bA)$ which usually gets lost.
An algebraic setting that takes account of this duality from
the outset is the one of an \emph{associative pair}
(see Appendix A and  \cite{BeKi09}).
For instance, the square matrix algebras $\gL_n(A;\K)$ are generalized
by the \emph{rectangular matrix algebras}
$\gL_{p,q}(A;\K):=M(p,q;\K)$ with bracket $[X,Y]_A$ where $A$ now belongs to
the ``opposite'' matrix space $M(q,p;\K)$.
In the pair setting, a map $\phi$ is an involution if and only if
so is $-\phi$, and hence $\Herm(\phi)$ and $\Aherm(\phi)$ simply
interchange their r\^oles if we replace $\phi$ by $-\phi$.
It is only the consideration of \emph{unit} or \emph{invertible elements} that may
break this symmetry: they may exist in one space but not in the
other.

\ssk
The following table summarizes the definition of classical
Lie algebras and their homotopes.
In the general linear cases, $\K$ may be any ring (in particular,
the quaternions $\HHH$ are admitted); in the orthogonal and symplectic
families $\K$ has to be a commutative ring, and for the unitary families
we use an involution of $\K$:  if $\K=\C$, we use usual
complex conjugation, and for $\K=\HHH$ we use the following
conventions: if nothing else is specified, we use the
``usual'' conjugation $\lambda \mapsto \overline \lambda$
 (minus one on the imaginary part
$\im \HHH$ and one on the center $\R \subset \HHH$).
If we consider $\HHH$ with its  ``split'' involution $\lambda \mapsto
\widetilde \lambda:=j \overline \lambda j^{-1}$, then we write $\widetilde \HHH$.
For instance,
$\Herm(n;\widetilde \HHH)$ is the space of quaternionic matrices
such that $\widetilde X = X^t$, and  $\uu_n(1 ;\widetilde \HHH)$
is the Lie algebra often denoted by $\so^*(2n)$.
In all cases, the Lie bracket is
$[X,Y]_A = XAY - YAX$.
Note finally that the trace map does not behave well with respect
to our contractions, and therefore we do not define homotopes of
special linear or special unitary algebras.
\msk

\noindent
\begin{tabular}{llll}
family name & label and space & parameter space & Lie bracket \cr
\hline
general linear (square)
&$\gL_n(A;\K):=M(n,n;\K)$
& $A \in M(n,n;\K)$
& $[X,Y]_A$
\cr
general linear (rectan.)
&
$\gL_{p,q}(A;\K):=M(p,q;\K)$
& $A \in M(q,p;\K)$
& $[X,Y]_A$
\cr
orthogonal
&
$\oo_n(A;\K):= \Asym(n;\K)$
& $A \in \Sym(n;\K)$
& $[X,Y]_A$
\cr
[half-] symplectic
&$\sP_{n/2}(A;\K):= \Sym(n;\K)$
& $A \in \Asym(n;\K)$
& $[X,Y]_A$
\cr
$\C$-unitary
&$\uu_n(A;\C):= \Aherm(n;\C)$
& $A \in \Herm(n;\C)$
& $[X,Y]_A$
\cr
$\HHH$-unitary
& $\uu_n(A;\HHH):= \Aherm(n;\HHH)$
& $A \in \Herm(n;\HHH)$
& $[X,Y]_A$
\cr
$\HHH$-unitary split
& $\uu_n(A;\widetilde \HHH):= \Aherm(n;\widetilde \HHH)$
& $A \in \Herm(n;\widetilde \HHH)$
& $[X,Y]_A$
\end{tabular}

\msk \nin
The expert reader will certainly have remarked that everything we
have said so far holds, {\sl mutatis mutandis}, for ``Lie'' replaced
by ``Jordan'': $\Herm(\bA,*)$ is a \emph{Jordan} algebra, and
in the \emph{Jordan pair} setting the r\^oles of $\Herm(\bA)$ and
$\Aherm(\bA)$ become more symmetric.
Indeed, a conceptual and axiomatic theory will use the Jordan- \emph{and}
Lie-aspects of an associative product in a crucial way
-- see remarks in Chapter 6 and in \cite{Be08c}. In order to keep this paper
accessible for a wide readership, no use of Jordan
theory will be made in this work.

\subsection{Homotopes of classical groups}

Now let us explain the main ideas serving to ``globalize'' the
Lie algebra situation just described.
First of all, for the classical Lie algebras introduced above
it is easy to define explicitly a corresponding
 algebraic group: in the setting of an abstract unital algebra
$\bA$ with Lie bracket $[x,y]_a=xay-yax$, one defines the set
$$
G(\bA,a) := \{ x \in \bA | \, 1 - xa \in \bA^\times \}
$$
and checks that
$$
x \cdot_a y := x + y - xay
$$
is a group law on
$G(\bA,a)$ with neutral element $0$ and inverse of $x$ given by
$$
j_a(x):= - (1 - xa)^{-1} x.
$$
It is easily seen (cf.\ Lemma \ref{Liealgebra})
that the Lie algebra of this group is given by the bracket $[x,y]_a$.
Next, observe that an involution $*$ of $\bA$ induces an isomorphism
from $G(\bA,a)$ onto the opposite group of $G(\bA,a^*)$. Therefore,
if $a$ is Hermitian, $*$ induces a group antiautomorphism of order $2$, and
we can define the \emph{$a$-unitary group} as usual to be the subgroup
of elements $g \in G(\bA,a)$ such that $g^* = j_a(g)$.
If $a$ is skew-Hermitian, the \emph{$a$-symplectic group} is defined
similarly by the condition $-g^* = j_a(g)$.
Specializing to the classical matrix algebras, we get the following
list of classical groups:
\msk

\noindent
\begin{tabular}{llll}
label &  underlying set  & parameter space & product \cr
\hline
 $\Gl_n(A;\K)$ & $:=\{ X \in M(n,n;\K) | 1 - AX \, \mbox{invertible} \}$
& $A \in M(n,n;\K)$
& $X \cdot_A Y$
\cr
$\Gl_{p,q}(A;\K)$ & $:=\{ X \in M(p,q;\K) | 1 - AX \, \mbox{invertible} \}$
& $A \in M(q,p;\K)$
& $X \cdot_A Y$
\cr
$\OO_n(A;\K)$ & $:= \{ X \in \Gl_n(A,\K) |  X + X^t = X^t AX \}$
& $A \in \Sym(n;\K)$
& $X \cdot_A Y$
\cr
$\Sp_{n/2}(A;\K)$ & $:= \{ X \in \Gl_n(A,\K) |  X - X^t = X^t AX \}$
& $A \in \Asym(n;\K)$
& $X \cdot_A Y$
\cr
 $\UU_n(A;\C)$ &
$:= \{ X \in \Gl_n(A,\K) |  X + \overline X^t = \overline X^t AX
\}$
& $A \in \Herm(n;\C)$
& $X \cdot_A Y$
\cr
 $\UU_n(A;\HHH)$ & $:= \{ X \in \Gl_n(A,\HHH) |
X + \overline X^t = \overline X^t AX
\}$
& $A \in \Herm(n;\HHH)$
& $X \cdot_A Y$
\cr
 $\OO_n(A;\widetilde \HHH)$ & $:= \{ X \in \Gl_n(A,\HHH) |
X + \widetilde X^t = \widetilde X^t AX
\}$
& $A \in \Herm(n;\widetilde \HHH)$
& $X \cdot_A Y$
\end{tabular}

\msk \nin
Finally, one may observe that this realization of classical groups has the advantage of leading to
a natural ``semigroup hull": e.g., if $A^t=A$, a direct computation shows that the set
$\hat \OO_n(A;\K):=\{ X \in M(n,n;\K) \vert \, X^t +X =X^t AX \}$
is stable under the product $\cdot_A$, which turns it into a semigroup with unit element
$0$, and similarly in all other cases.

\subsection{``Projective'' theory of classical torsors}

The definition of the classical groups given
above is useful for calculating their Lie algebras and for starting
to analyze their group structure (and their topological structure if
$\K$ is a topological field or ring), but also has several
drawbacks: firstly, note that the product $X \cdot_A Y$
is affine in both variables, and hence our groups are realized
as subgroups of the affine group of the matrix space $M(n,n;\K)$.
The corresponding \emph{linear representation} in a space of dimension
$n^2 + 1$ is not very natural, and one may wish to realize these
groups in more natural linear representations.
Secondly,
whereas
 the general linear groups are, for all $A$, realized as (Zariski-dense)
parts of a common ambient space ($M(n,n;\K)$, resp.\ $M(p,q;\K)$),
this is not the case for the other classical groups:
the underlying set depends on $A$, and hence the realization
is not adapted to the point of view of deformations or contractions.
Finally, and related to the preceding item, one has the impression that
the ``semigroup hull" $\hat \OO_n(A;\K)$ depends on the realization,
and that it should rather be part of some maximal semigroup
hull  intrinsically associated to the group $\OO_n(A;\K)$ .

\ssk
In the present work, we will give another
realization of the classical groups (and, much more generally,
of the groups attached to abstract involutive algebras) having
none of these drawbacks: it is a sort of
projective realization, as opposed to the affine picture just given.
In a first step, we get rid of base points in groups by
considering them as \emph{torsors}, that is, we work with
the ternary product $(xyz):=xy^{-1}z$ of a group.
By \emph{classical torsor} we simply mean  a classical group
from the preceding table equipped with this ternary law, i.e.,
by forgetting their base points.
For the general linear family, we have seen in Part I of this
work that there is a common realization of all groups
$\Gl_{p,q}(A,\K)$ inside the \emph{Grassmannian}
$\cX:=\Gras(\K^{p+q})$ in such a way that they
are realized as subgroups of the projective group $\PP \Gl(p+q,\K)$.
The parameter space is again the complete space $\cX$, and
``space'' and ``parameter'' variables are incorporated into
a single object (called an \emph{associative geometry}, given by
a pentary product map $\Gamma:\cX^5 \to \cX$)
having surprising properties.
In the present work we show that,
for the other families, there is a more refined construction,
relying on the existence of \emph{involutions} (antiautomorphisms of
order 2)
of associative geometries.
For the classical groups, these involutions are orthocomplementation
maps, so that the fixed point spaces are \emph{varieties of
Lagrangian subspaces}.
We will realize all orthogonal groups as (Zariski dense) subsets
of the Lagrangian variety of a quadratic form of signature
$(n,n)$, and the [half-] symplectic groups in the Lagrangian variety
of a symplectic form on $\K^{2n}$.
 The underlying Lagrangian
variety plays the r\^ole of a ``projective completion'' of these
groups (also called ``projective compactification'' if $\K=\R$ or $\K=\C$ since it is
compact in these cases), and in particular we will show that
the group law extends to a semigroup law on the projective completion,
thus defining the intrinsic and maximal (compact) semigroup hull
for all classical groups and their homotopes.
As in the general linear case, this achieves a realization in
which all ``deformations'' or ``contractions'' are globally defined
on the space level.
In contrast to the general linear case, the parameter space now
is different from the underlying Lagrangian variety of the
group spaces: it is another Lagrangian variety which we call
the \emph{dual Lagrangian}. This duality reflects the duality
between $\Herm(\bA,*)$ and $\Aherm(\bA,*)$ mentioned in Section 0.1.

\subsection{Contents}

The contents of this paper is as follows:
in Chapter 1 we recall basic facts on the ``general linear construction'';
in Chapter 2 we define and construct involutions of associative
geometries: in Theorem \ref{Th:3.2} we prove that orthocomplementation maps
of non-degenerate forms are involutions; in Chapter 3 we
 describe the ``projective" construction of torsors and groups
associated to (restricted) involutions of associative geometries
(Lemma  \ref{clue}), their tangent objects with respect to various
choices of base points (Theorems \ref{unitalMaintheorem}
 and \ref{functor'}) as well as the link with the ``affine" realization given
 above (Theorem \ref{conjug}). In Chapter 4 we present the classification
 of homotopes of classical groups (over $\K=\R$ or $\C$, the case of general base fields
 or rings being at least as complicated as the problem of classifying involutive associative
 algebras, see \cite{KMRS98}). In
Chapter 5 we describe the semi-group completion of classical groups
(Theorem \ref{semitorsor}); the main difficulty here is to prove that
non-degenerate forms induce involutions of geometries in a
``strong'' sense. This requires some investigation of the linear
algebra of linear relations, complementing those from Chapter 2 of
Part I of
this work, and which may be of interest in its own right.
Finally, in Chapter 6 we give some brief comments on a possible axiomatic approach,
involving both the Jordan- and the Lie side of the whole structure, and
Appendix A contains the relevant definitions on involutions of associative pairs.

\subsection{Related work}

Finally, let us add some words on related literature.
It seems to be folklore in symplectic geometry that the group law
of $\Sp(m,\R)$ extends to the whole Lagrangian variety if we interpret it via
\emph{composition of linear relations}: the composition of two
Lagrangian linear relations is again Lagrangian
(see appendix on ``linear symplectic reduction'' in
\cite{CDW87} or Theorem 21.2.14 in \cite{Ho85}).
In a case-by-case way, Y.\ Neretin (\cite{Ner96}) has given
similar constructions for other families of complex or real Lagrangrian
varieties (``categories $B$, $C$, $D$'', see loc.\ cit., p.\ 85 ff and
loc.\ cit.\ Appendix A for their real analogs).
It would be very interesting to investigate
further the relationship between our work and Neretin's, in particular
in view of applications in harmonic analysis and quantization.
Note that Neretin in loc.\ cit.\ p.\ 59 uses
a modified composition law of linear relations
in order to obtain a jointly continuous
operation; since we do not consider topologies here, we leave
the [important] topic of joint continuity for later work.

\begin{notation}
Throughout this work, $\bK$ denotes  a commutative unital ring and
$\bB$ an associative unital $\bK$-algebra, and we will consider
\emph{right} $\bB$-modules $V,W,\ldots$.
We think of $\bB$ as ``base ring'', and the letter $\bA$ will be
reserved for other associative $\bK$-algebras such as $\End_\bB(W)$.

If $V = a \oplus b$ is a direct sum decomposition of a vector space
or module, we denote by $P^a_b:V \to V$ the projection with kernel
$a$ and image $b$.
\end{notation}

\section{The general linear family}

\subsection{Groups and torsors living in Grassmannians} \label{sec:1.1}

We are going to recall the basic construction from \cite{BeKi09}
which realizes groups like $\Gl_n(A,\K)$ inside a Grassmannian manifold.
Let $W$ be a right $\B$-module and
$\XX = \Gras(W)$ be the Grassmannian of all right $B$-submodules
of $W$. A pair $(x,a) \in \XX^2$ is called \emph{transversal}
(denoted by $a \top x$ or $x \top a$) if $W=x \oplus a$.
The set of all complements of $a$ is denoted by $C_a$, so that
$$
C_{ab}:=C_a \cap C_b
$$
is the set of common complements of $a$ and $b$. One of the main
results of \cite{BeKi09}  says that the set $C_{ab}$ carries two
canonical torsor-structures. More precisely, we define,
for $(x,a,b,z) \in \cX^4$ such that $a \top x$, $b \top z$,
the endomorphism of $W$
\begin{equation}
M_{xabz}:=P^a_x - P_b^z = P^a_x - 1 + P_z^b . \label{1.1}
\end{equation}
By a direct calculation  (see  \cite{BeKi09}, Prop.\ 1.1), one
sees that
\begin{equation}
M_{xabz} = M_{zbax}, \quad M_{xabz} = - M_{axzb}, \label{1.2}
\end{equation}
and, if $x,z \in U_{ab}$, then $M_{xabz}$ is invertible with inverse
\begin{equation}
(M_{xabz})^{-1} = M_{zabx} = M_{xbaz} . \label{1.3}
\end{equation}
Recall  (see, e.g., \cite{BeKi09}) that a \emph{torsor} is the base point-free
version of a group (a set $G$ with a ternary map $G^3 \to G$,
$(xyz) \mapsto (xyz)$ such that $(xyy)=x=(yyx)$ and
$(xy(zuv))=((xyz)uv)$).
Then (\cite{BeKi09}, Th.\ 1.2):

\begin{theorem}
\thmlabel{geom_props}
\begin{enumerate}[label=\roman*\emph{)},leftmargin=*]
\item For $a,b \in \cX$ fixed, $C_{ab}$ with product
\[
(xyz):= \Gamma(x,a,y,b,z):=M_{xabz}(y)
\]
is a torsor (which will be denoted by $U_{ab}$).
In particular, for all $y \in C_{ab}$, the set
$C_{ab}$ is a group with unit $y$ and multiplication
$xz=\Gamma(x,a,y,b,z)$. 
\item
$U_{ab}$ is the opposite torsor of $U_{ba}$
(same set with reversed product):
\[
\Gamma(x,a,y,b,z)=\Gamma(z,b,y,a,x)
\]
 In particular,
the torsor $U_a:=U_{aa}$ is commutative.
\item The commutative torsor $U_a$ is the underlying additive torsor
of an affine space:
$U_a$ is an affine space over $\bK$, with
additive structure given by
\[
x+_y z = \Gamma(x,a,y,a,z),
\]
(sum of $x$ and $z$ with respect to the origin $y$),
and  action of scalars  given by
\[
\Pi_s(x,a,y):= sy + (1-s)x = (s P^x_a + P^a_x) (y)
\]
(multiplication of $y$ by $s$ with respect to the origin $x$).
\end{enumerate}
\end{theorem}

\begin{definition} {\bf (The restricted multiplication map)}
We call \emph{restricted multiplication map} the map
$\Gamma:D_5 \to \XX$, defined on the \emph{set of admissible $5$-tuples}
$$
D_5:=\{(x,a,y,b,z) \in \XX^5 | \, x,y,z \in C_{ab} \}  ,
$$
by the formula from part i) of
the preceding theorem.
\end{definition}

\begin{definition} {\bf (Base points and tangent spaces)}
A \emph{base point} in $\cX$ is a fixed transversal pair,
usually denoted by $(o^+,o^-)$.
The \emph{tangent space at $(o^+,o^-)$} is the pair
$$
(\bA^+,\bA^-):= (C_{o^-},C_{o^+}) .
$$
Note that $(\bA^+,\bA^-)$
is a pair of $\K$-modules (with origin $o^\pm$ in $\bA^\pm$),
isomorphic to
\begin{equation}
\bigl(\Hom_\bB(o^+,o^-),\Hom_\bB(o^-,o^+) \bigr).
\eqnlabel{affine}
\end{equation}
This tangent space carries the structure of an \emph{associative
pair}  given by trilinear products (see \cite{BeKi09}, Th.\ 1.5)
\begin{equation}
\bA^\pm \times \bA^\mp \times \bA^\pm \to \bA^\pm, \quad
(u,v,w) \mapsto \langle u,v,w\rangle^\pm:=\Gamma(u,o^+,v,o^-,w).
\end{equation}
\end{definition}

\begin{definition} {\bf (Transversal triples)}
A \emph{transversal triple} is a triple of mutually transverse
elements. If we fix such a triple, we usually denote it by
$(o^+,e,o^-)$. In this case, $\bA:=C_{o^-}$ carries the
structure of an associative algebra with origin $o:=o^+$ and unit $e$, called the
\emph{tangent algebra at $o^+$ corresponding to the base triple
$(o^+,e,o^-)$}, with product
\begin{equation}
\bA \times \bA \to \bA, \quad
(u,v) \mapsto \Gamma(u,o^+,e,o^-,v).
\label{asproduct}
\end{equation}
In a dual way, $C_{o^+}$ is turned into an algebra with origin $o^-$. Both algebras
are canonically isomorphic via the inversion map $j=M_{eo^+o^-e}$.
\end{definition}

\subsection{Lie algebra and structure of the torsors $U_{ab}$}
\label{sec:1.2}
We explain the link between the torsors $U_{ab}$ and the groups
$\Gl_{p,q}(A;\bB)$ defined in the Introduction, as well as the
computation of their ``Lie algebra".

\begin{lemma} \label{Liealgebra}
Choose an origin $o^+$ in $U_{ab}$ and an element $o^- \top o^+$.
Then
the Lie algebra (in a sense to be explained in the following
proof) of the group $(U_{ab},o^+)$ is the ``tangent space''
$\bA^+=\Hom_\bB(o^+,o^-)$ with Lie bracket
$$
[X,Y] = X (a-b) Y - Y (a-b) X
$$
(note that $o^+ \in U_{ab}$ means that $a,b \in C_{o^+}=\bA^-$, so that
$a-b \in \bA^-$).
In particular, choosing $o^-=b$, we get the Lie algebra of $U_{A0}$:
$$
[X,Y] = XAY - YAX.
$$
\end{lemma}

\begin{proof}
The Lie algebra can be defined in a purely algebraic way, without
using ordinary differential calculus, as follows.
Let $T\K:=\K[\eps]:=\K[X]/(X^2)$,
$\eps^2=0$ be the ring of dual numbers over $\K$
and $TT\K:=T(T\K):=(\K[\eps_1])[\eps_2]$ be the ``second
order tangent ring''. Then $(\cX,\Gamma)$ admits scalar extensions
from $\K$ to $T\K$ and to $TT\K$, and the commutator in the
second scalar extension of the group $U_{ab}$ gives rise to
the Lie bracket in the way described in  \cite{Be06}, Chapter V.
This construction is intrinsic and does not depend on ``charts''.
Therefore we may choose $o^- :=b$ in order to simplify calculations
(the first formula from the claim then follows from the second one).
Then $U_{ab}=C_a \cap C_{o^-}=C_a \cap \bA^+$, and according
to \cite{BeKi09}, Section 1.4 we have the following ``affine picture''
of the group $(U_{ab},o)$: if,
under the isomorphism \eqnref{affine}, $a$ corresponds to the element
$A \in \bA^-=\Hom_\bB(o^-,o^+)$, then
 $U_{ab}$ corresponds to the set
\begin{equation}
U_{A0}=\{ X \in \Hom_\bB(o^+,o^-) | \,
1 - AX \, \mbox{is invertible in } \, \End_\bB(o^+)  \}
\label{U}
\end{equation}
with group law given by the product $Z \cdot_A X$ defined in the Introduction:
\begin{equation}
X \cdot Z = X + Z - ZAX . \eqnlabel{product}
\end{equation}

Since Formulas (\ref{U}) and
\eqnref{product} are algebraic, we may now determine explicitly the
\emph{tangent group}  of  $U_{0A}$ via
scalar extension by dual numbers: the operator
$$
1 - (A + \eps A')(X+\eps X') = 1 - AX + \eps(A'X + AX')
$$
is invertible iff so is
$1 - AX$, hence the tangent bundle $T(U_{0A})$ is
$U_{0A} \times \eps \Hom_\bB(o^+,o^-)$, with semidirect product group
structure
\begin{equation}
(X,\eps X') \cdot (Z,\eps Z') = \bigl(X + Z - ZAX, \eps (X'+Z'+
Z'AX + ZAX')\bigr) . \eqnlabel{product'}
\end{equation}
Repeating the construction, we obtain the second tangent bundle
$TT(U_{0A})$ by scalar extension from $\K$ to the ring
$TT\K$. As explained in \cite{Be06}, the Lie bracket
$[X,Y]$ arises from the commutator in the second tangent group via
$$
\eps_1 \eps_2 [X,Y]=
(\eps_1 X) (\eps_2 Y) (\eps_1 X)^{-1} (\eps_2 Y)^{-1} .
$$
A direct calculation, based on \eqnref{product'}, yields
$$
(\eps_1 X) (\eps_2 Y) =\eps_1 X + \eps_2 Y +
\eps_1 \eps_2 YAX,
$$
which, after a short calculation using that $(\eps_1  X)^{-1} = \eps_1  (-X)$,
$(\eps_1  Y)^{-1} = \eps_1  (-Y)$,
implies the claim.
\end{proof}

As is easily seen from the explicit formulas given above by choosing for $A$
special (idempotent) elements (cf.\  \cite{Be08b}), the groups $U_{ab}$ and their Lie
algebra have a \emph{double fibered structure}. These and related features for
symmetric spaces will be investigated in \cite{BeBi}.

\section{Construction of involutions}

\subsection{Definition of (restricted) involutions}
Whenever in a
category we have for each object $\XX$  a canonical notion of
an ``opposite object'' $\XX^{op}$, there is a natural
notion of \emph{involution}. This is the case for
groups, torsors or associative geometries.

\begin{definition}
A \emph{restricted involution} of the Grassmannian geometry
$\cX=\Gras(W)$ is a bijection $f:\cX \to \cX$ of order two and such that
\begin{enumerate}
\item
$f$ preserves transversality:
 for all $a,x \in \cX$:
$a \top x$ iff $f(a) \top f(x)$,
\item
$f$ is an isomorphism onto the opposite restricted product map:
 for all $5$-tuples $(x,a,y,b,z)$ such that $x,y,z \in U_{ab}$,
$$
f \big(\Gamma(x,a,y,b,z) \big)=
\Gamma(f x,f b,f y,f a,f z)=
\Gamma(f z,f a,f y,f b,f x).
$$
\item
$f$ induces affine maps on affine parts:
 for all $3$-tuples $(x,a,y)$ such that $x,y \top a$, and $r \in \K$,
$$
f \bigl( \Pi_r(x,a,y)) = \Pi_r\bigr(f x,f a,f y \bigl).
$$
\end{enumerate}
In other words, by (1), $f$ induces well-defined restrictions
$U_{ab} \to U_{f(a),f(b)}$ and $U_a \to U_{f(a)}$, which induce, by (2), anti-isomorphisms of
torsors $U_{ab} \to U_{f( a),f (b)}$, and by (3), isomorphisms of affine spaces
$U_a \to U_{f(a)}$.

The fixed point space $\cY:=\cX^\tau$ of an involution $\tau$
will be called the \emph{Lagrangian type geometry of $(\cX,\tau)$} (if it
is not empty).
\end{definition}

In general, nothing guarantees \emph{existence} of restricted involutions.
Before turning to the general theory (next chapter), we will show that under certain
conditions one can construct them by using bilinear or
sesquilinear forms. In these cases, $\cY$ will be indeed realized as
a geometry of Lagrangian subspaces.

\subsection{Non-degenerate forms and adjoinable pairs} \label{Sec:3.1}

We assume that our $\B$-module $W$ admits a non-degenerate
sesquilinear form
$$
\beta :W \times W \to \B.
$$
By sesquilinearity we mean $\beta(vr,w)=\overline r \beta(v,w)$,
$\beta(v,wr)=\beta(v,w)r$ for $v,w \in W$, $r\in \B$, where
$$
\B \to \B, \quad z \mapsto \overline z
$$
is some fixed involution (antiautomorphism of order $2$) of $\B$,
and non-degeneracy means that $\beta(v,W)=0$ or $\beta(W,v)=0$
implies $v=0$. Of course, for $\B=\K$ and $\overline z=z$ we
get bilinear forms.
Moreover, we assume that $\beta$ is Hermitian or
skew-Hermitian:
$$
\forall v,w \in W:\,
\beta(v,w)=\overline{\beta(w,v)},
\quad {\rm resp.} \quad \forall v,w \in W:\,
\beta(v,w)=-\overline{\beta(w,v)}.
$$
As usual, the orthogonal complement of a subset $S \subset W$
will be denoted by $S^\perp$. The orthogonal complement of a right
submodule is again a right submodule, but, unfortunately,
it is in general not true that the orthocomplementation map
$\perp:\cX \to \cX$ satisfies the properties of a (restricted) involution:
in general, it does not even preserve transversality, nor is it of order two.

\begin{definition}
A pair $(x,a) \in \cX \times \cX$ is called \emph{adjoinable} if
$W = x \oplus a$ and $W=x^\perp \oplus a^\perp$.
\end{definition}

\begin{lemma}
A pair $(x,a) \in \cX \times \cX$ is adjoinable if and only if
 the projection $P:=P^a_x$ is adjoinable; i.e., there exists a linear operator
$P^*:W \to W$ such that
\begin{equation}
\forall v,w \in W: \quad \quad \beta(v,Pw)=\beta(P^*v,w).
\label{3.2a}
\end{equation}
Moreover, in this case we have $(x^\perp)^\perp =x$
and $(a^\perp)^\perp =a$.
\end{lemma}

\begin{proof}
Assume $P^*$ exists.  If two operators $f,g$ are
adjoinable, then we have $(gf)^*=f^*g^*$, and hence $P^*$ is
again idempotent. Moreover, the kernel of $P^*$ is
  $\ker P^* = (\im P)^\perp = x^\perp$.
Now, $P$ is adjoinable if and only if so is $Q:=1-P$, whence
 $\im P^* =\ker Q^* = a^\perp$,
and thus $W = x^\perp \oplus a^\perp$. Moreover, this shows that
\begin{equation}
(P_x^a)^* = P_{a^\perp}^{x^\perp} . \label{3.2}
\end{equation}
Reversing these arguments, we see that, if $(x,a)$
is adjoinable, equation (\ref{3.2}) defines an operator $P^*$,
and a direct check shows that then (\ref{3.2a}) holds.
Moreover, from $(P^*)^* =P$ the relations $(x^\perp)^\perp =x$
and $(a^\perp)^\perp =a$ follow.
\end{proof}

The lemma shows that, in the general case, we should not work with
the full Grassmannian, but only with its adjoinable elements.
For simplicity, let us first look at a case where the Grassmannian
is well-behaved, namely the case $W = \bB^n$:

\begin{theorem} {\bf (Construction of involutions: case of $\bB^n$)} \label{Th:3.2}
Let $W = \bB^n$ and $\cX$ be the Grassmannian of all right submodules
that admit some complementary right submodule, and let $\beta$ be
a non-degenerate Hermitian or skew-Hermitian form on $\bB$.
Then the orthocomplementation map
$$
\perp_\beta : \cX \to \cX, \quad x \mapsto x^\perp
$$
is a restricted involution of $\cX$.
\end{theorem}

\begin{proof}
For $W=\B^n$,  every non-degenerate sesquilinear form is given by
$$
\beta(x,y) = \sum_{i,j=1}^n \overline x_i b_{ij} y_j
$$
with some invertible matrix $B=(b_{ij})$. By assumption, $B$ is
Hermitian or skew-Hermitian.
As can be checked by
a direct matrix calculation,
in this case every linear operator $X:W \to W$ is adjoinable,
with adjoint given by the adjoint matrix $X^*$ of $(X_{ij})$:
$$
X^* = B^{-1} \overline X^t B
$$
where $X^t$ is the transposed matrix of $X$. In particular, if $x$ is an
arbitrary complemented right-submodule of $\B^n$ with complement
$a$, then $P:=P^a_x$ is adjoinable. Thus every transversal pair
$(x,a)$ is adjoinable, and moreover
$$
x^\perp = \im(P)^\perp = \ker(P^*) .
$$
We have thus shown that the orthocomplementation map is of order
two and preserves
transversalilty. In order to prove the crucial property
\begin{equation}
\Gamma(z^\perp,a^\perp,y^\perp,b^\perp,x^\perp)=
\big( \Gamma(x,a,y,b,z) \big)^\perp
\label{inv}
\end{equation}
we observe that, for all $x \in \Gras(W)$ and all linear maps
$F:W \to W$
\begin{equation} \label{Fperp}
(Fx)^\perp = (F^*)\inv (x^\perp)
\end{equation}
(inverse image), and if $F$ is bijective, $(F^*)\inv = (F\inv)^*$ (inverse map).
We apply this to the bijective map $F = M_{xabz}$ (for $x,z \in C_{ab}$) whose inverse is
$F\inv = M_{zabx}=M_{xbaz}$ and whose adjoint can be computed using
(\ref{3.2}):  for $x,z \in C_{ab}$, the
operator $M_{xabz}$ has an adjoint given by
\begin{equation}
(M_{xabz})^* = (P^a_x - P^z_b)^* =M_{a^\perp x^\perp z^\perp b^\perp} =
-M_{x^\perp a^\perp b^\perp z^\perp} . \label{3.3}
\end{equation}
Now let $a,b \in \XX$ and
$x,y,z \in C_{ab}$. Then, with $F = M_{xabz}$,
\begin{eqnarray*}
\big( \Gamma(x,a,y,b,z) \big)^\perp & = &
\big( F(y) \big)^\perp = (F^*)\inv y^\perp
\cr
&=& M_{x^\perp a^\perp b^\perp z^\perp}\inv (y^\perp)
= M_{x^\perp b^\perp a^\perp z^\perp} (y^\perp) =
 \Gamma(x^\perp,b^\perp,y^\perp,a^\perp,z^\perp) \, .
\end{eqnarray*}
This proves (\ref{inv}).
Finally, property (3) of an involution can be proved in the same
way as (\ref{inv}) (and this property is already known since it depends only
on the underlying Jordan structure, see, e.g.,  \cite{Be04}).
\end{proof}

The cases $n=1$ and $n=2$ of the preceding result deserve special interest.
For $n=1$, we work with the form $\beta(u,v)=\overline u \, v$,
and we consider the Grassmannian of complemented right ideals
in $\B$ with involution $\ker e \mapsto \im \, \overline e$
(where $e \in \B$ is an idempotent,
$\ker e = (1-e) \B$, $\im e = e \B$).
The case $n=2$ enters in the proof
of Theorem \ref{functor2} (next chapter).

\subsection{The adjoinable Grassmannian}

As we will see in Theorem \ref{functor2}, the case $n=2$ is already
suitable to treat all seemingly more general cases.
Returning thus to the case of a general $\bB$-module $W$ with
a non-degenerate Hermitian or skew-Hermitian form $\beta$, we may
proceed as follows:
let
$$
\bA := \{ f \in \End_\bB(W) | \, \exists f^*  \in \End_\bB(W) : \,
\forall v,w \in W: \,\beta(v,fw)=\beta(f^*v,w) \}
$$
the set of all adjointable linear operators. Then $\bA$ is
a subalgebra of $\End_\bB(W)$, and $*$ is an involution on $\bA$.
Now define the \emph{adjoinable Grassmannian of $\beta$} to be
$$
\cX_\beta := \{ \im \, P | \, P \in \bA, P^2 = P \},
$$
the set of all submodules $x$ admitting  a complement $a$ such that
the projection $P:=P^a_x$ is adjointable.
(In general, not all submodules have this property --
consider e.g.\ a dense proper
subspace $x$ in a Hilbert space.)
Let $\tilde \cX :=\{  P \bA | \, P \in \bA, P^2 = P \}$
be the Grassmannian of all complemented right modules in $\bA$.
Then the map
$$
\tilde \cX \to \cX_\beta, \quad P \bA \mapsto \im P
$$
is well-defined, bijective and compatible with the structure
maps $\Gamma$. We use it to push down
 $\tau$  to an involution of $\cX_\beta$, so that we can carry
out all preceding constructions on the adjoinable Grassmannian.

\section{Groups and torsors associated to involutions}

We assume, for all of this chapter, that  $\tau:\XX \to \XX$ is a restricted
involution
of the Grassmannian geometry $\XX = \Gras_\BB(W)$ and write
$\cY$ for its fixed point space.
There are two different ways to construct groups and torsors
associated to $(\cX,\tau)$. Here is the first construction,
which simply mimics the usual definition of unitary and orthogonal
groups:

\begin{definition}
Fix three points $a,o,b \in \cY$ such that $o \in U_{ab}$, considered
as origin  in the group $(U_{ab},o)$,
and let $x^{-1}:=M_{oabo}(x)$ be inversion in this group.
Then $\tau$ induces an antiautomorphism of this group:
\begin{eqnarray*}
\tau(xy) &  =&  \tau \Gamma(x,a,o,b,y)=
\Gamma(\tau(y),\tau(a),\tau(o),\tau(b),\tau(x)) \cr
&=&
\Gamma(\tau(y),a,o,b,\tau(x))=
\tau(y)\tau(x),
\end{eqnarray*}
and hence
$$
\UU(\tau;a,o,b):=\{ x \in U_{ab} \vert \, \tau(x) = x^{-1} \}
$$
is a subgroup, called the \emph{$\tau$-unitary group (located at $(a,o,b)$)}.
\end{definition}

This group is not a subset of the Lagrangien geometry $\cY$,
but rather is ``tangent" to the ``antifixed space of $\tau$": indeed, the
differential of inversion at $o$ is the negative of the identity,
and hence the tangent space of $\UU(\tau;a,o,b)$ at the identity
should be the minus one eigenspace of $\tau$. This will be
made precise below (Theorem \ref{conjug}). Next, we describe a second construction
of groups having the advantage that it directly leads to torsors
living in the Lagrangian geometry:

\begin{lemma}\label{clue} Let $\tau:\XX \to \XX$ be a restricted
involution
of the Grassmannian geometry $\XX = \Gras_\BB(W)$
and denote by $\YY:=\XX^\tau$ its Lagrangian type geometry.
Then
\begin{enumerate}[label=\roman*\emph{)},leftmargin=*]
\item
 for any $a \in \XX$,
$\tau$ induces a torsor-automorphism of the torsor
$U_{a,\tau(a)}$. In particular, the fixed point set
$$
\GG(\tau;a) := (U_{a,\tau(a)})^\tau =U_{a,\tau(a)} \cap \YY
$$
is a subtorsor of $U_{a,\tau(a)}$.
\item
As a set,
$\GG(\tau, a)= U_{a} \cap \YY$.
\item
$\GG(\tau , \tau(a))$ is the opposite torsor of $\GG(\tau,a)$.
If $a  \in \cY$, then the torsor $\GG(\tau,a)$ is abelian, and it
is the underlying additive torsor of an affine space over
$\K$.
\end{enumerate}
\end{lemma}

\begin{proof} (i)
Note first that $x \in C_{a,\tau(a)}$ if and only if
$\tau(x) \in C_{\tau(a),\tau^2(a)}=C_{a,\tau(a)}$ since
$\tau$ preserves transversality and is of order $2$.
Next we show that $\tau$ preserves the torsor law
$(xyz)_a =\Gamma(x,a,y,\tau(a),z)$
of $U_{a,\tau(a)}$:
\begin{eqnarray*}
\tau (((xyz)_a) &=&
\tau(\Gamma(x,a,y,\tau(a),z))=
\Gamma(\tau z,\tau a,\tau y,a,\tau x) \cr
&=&
\Gamma(\tau x, a,\tau y,\tau a,\tau z) =
(\tau x \, \tau y \, \tau z)_a.
\end{eqnarray*}
Clearly, the fixed point space $U_{a,\tau(a)} \cap \cY$ is then a subtorsor.

(ii)  If $x \in \cY$, i.e., $\tau (x)=x$, then  $x \top a$ is equivalent
to $x \top \tau(a)$, whence
$$
\YY \cap U_a = \YY \cap U_a \cap U_{\tau(a)} =
\YY \cap U_{a,\tau(a)} .
$$

(iii) $U_{a ,\tau(a)}$ is the opposite torsor of
$U_{\tau(a), a}$. If $a=\tau(a)$, then the arguments given above show that
$\tau$ is an automorphism of order $2$ of the affine
space
$U_a$ and hence its fixed point space is an affine subspace.
\end{proof}

In order to compare both constructions, we have to to study the behaviour of involutions
with respect to basepoints.

\subsection{Basepoints, and the dual involution}
Let us fix a
 base point  $(o^+,o^-)$ in $\XX$. Recall from \cite{BeKi09}, Th.\ 1.3,
that  the middle multiplication operator  $M_{o^+ o^- o^- o^+}$
is an automorphism of $\Gamma$. By (\ref{1.3}), it is invertible
and equal to its own inverse. Moreover,
$$
M_{o^+ o^- o^- o^+} (o^\pm)= \Gamma(o^+ ,o^-,o^\pm, o^- ,o^+)=
o^\pm.
$$
Thus $M_{o^+ o^- o^- o^+}$ is a base point preserving automorphism
of the Grassmannian geometry. Its effect on the additive groups
$\bA^\pm$ is simply inversion, that is, multiplication by the scalar
$-1$.

\begin{definition}
A (restricted) involution $\tau$ of $\XX$ is called
\begin{itemize}
\item
\emph{base point preserving} if $\tau(o^+)=o^+$ and
$\tau(o^-)=o^-$, and
\item
\emph{base point exchanging} if $\tau(o^+)=o^-$ and
$\tau(o^-)=o^+$.
\end{itemize}
\end{definition}

\begin{lemma} \label{dual}
Assume $\tau$ is a base point preserving or base point exchanging
involution of $\XX$. Then $\tau$ commutes with the automorphism
 $M_{o^+ o^- o^- o^+}$, and
$$
\tau':= M_{o^+ o^- o^- o^+} \circ \tau =
\tau \circ M_{o^+ o^- o^- o^+}
$$
is again of the same type (base point preserving, resp.\  exchanging
involution) as $\tau$. \end{lemma}

\nin We call $\tau'$ the  \emph{dual involution} (denoted
by $-\tau$ in a context where $(o^+,o^-)$ is fixed).

\begin{proof}
Thanks to the symmetry relation
$M_{xabz}=M_{axzb}$ we get in either case
$$
\tau \circ  M_{o^+ o^- o^- o^+} \circ \tau =  M_{\tau o^+,\tau o^-,\tau o^-, \tau o^+} =
M_{o^+ o^- o^- o^+}.
$$
Therefore $\tau'$ is again of order $2$, and it is an antiautomorphism
having the same effect on $o^\pm$ as $\tau$ since  $M_{o^+ o^- o^- o^+}$
is base point preserving.
\end{proof}

Recall from \cite{BeKi09} that, with respect to a fixed base point $(o^+,o^-)$
and $a \in \bA^-$,
$$
\tilde t_a : = M_{o^+ a o^- o^+} \circ M_{o^+ o^- o^- o^+}  =
M_{a o^+ o^+ o^-} \circ  M_{o^- o^+ o^+ o^-}
=L_{a o^+ o^- o^+}
$$
is the (left) translation operator defined by $a$
 in the abelian group $U_{o^+} \cong \bA^-$.
It acts rationally on $\bA^+$ by  the so-called \emph{quasi inverse map}.

\begin{theorem}\label{conjug}
Assume $\tau$ is a base point preserving involution of $\XX$ and let
$a \in \cY \cap U_{o^-} = (\bA^+)^\tau$.
Then the groups $\GG(-\tau;a)$ and $\UU(\tau;2a,o^+,o^-)$ are isomorphic
(the multiple $2a=a+a$ taken in $\bA^+$).
An isomorphism is induced by $\tilde t_a$.
\end{theorem}

\begin{proof}
Having fixed the base point, we use the notation $- \id:= M_{o^+ o^- o^- o^+}$.
We have to show that the group $U_{a,-a}$ with its automorphism $\tau'$ is conjugate
to the group $U_{2a,o^-}$ with its automorphism $i_{2a} \tau$ where $i_{2a}:= M_{o^+ 2a \, o^- o^+}$
 is inversion in the group $(U_{2a,o^-},o^+)$.
 First of all,
 $$
 \tilde t_a(a)=a+a=2a, \quad \tilde t_a(-a)=a+(-a)=o^-
 $$
 (sums in $(\bA^-,o^-)$), hence $\tilde t_a$ induces a torsor isomorphism from
 $U_{a,-a}$ onto $U_{2a,o^-}$ preserving the base point $o^+$.
 Next, observe  that
$$
i_{2a} \circ (-\id )  =  M_{o^+ 2a o^- o^+} \circ  M_{o^+ o^- o^- o^+} =
\tilde t_{2a}
$$
whence, using that $\tau' \circ \tilde t_a = \tilde t_{\tau' a} \circ \tau' = \tilde t_{-a} \circ \tau'$,
$$
\tilde t_{-a} \circ i_{2a} \tau \circ \tilde t_{a} =
\tilde t_{-a} \circ \tilde t_{2a} \circ (-\id) \circ  \tau \circ \tilde t_{a} =
\tilde t_{-a} \circ \tilde t_{2a}  \tau' \circ \tilde t_{a} =
\tilde t_{-a}  \tilde t_{2a}  \tilde t_{-a} \circ  \tau'
= \tau'
$$
where the last equality follows from  the relation $\tilde t_b \tilde t_c = \tilde t_{b+c}$.
\end{proof}

In the affine chart $\bA^+$, $\tilde t_a$ acts as a birational map, transforming the
affine realization $\UU(\tau;2a,o^+,o^-)$ to a rational realization that is
Zariski-dense in $(\bA^+)^{-\tau}$.
If $2$ is invertible in $\K$, all $\tau$-unitary groups $\UU(\tau; b,o,c)$
have such a realization $\GG(\tau';a)$  (just choose the base point $(o^+,o^-)=(o,c)$ and let
$a:= b/2$). If $2$ is not invertible in $\K$, such a realization is not always possible.

Concerning \emph{involutions of associative pairs} and \emph{associative triple systems}, to be used in the
following result, see Appendix A.

\begin{theorem}\label{functor1}
Assume $\tau$ is a restricted involution of the Grassmannian geometry
$\XX $, and let
$(\AAA^+,\AAA^-)$ be the associative pair corresponding
to a base point $(o^+,o^-)$.
\begin{enumerate}[label=\roman*\emph{)},leftmargin=*]
\item
If $\tau:\XX \to \XX$ is base-point preserving,
 then by restriction $\tau$ induces $\K$-linear
maps $\tau^\pm:\AAA^\pm \to \AAA^\pm$ which form a type
preserving involution of $(\AAA^+,\AAA^-)$.
\item
If $\tau:\XX \to \XX$ is base-point exchanging,
 then by restriction $\tau$ induces $\K$-linear
maps $\tau^\pm:\AAA^\pm \to \AAA^\mp$ which form a type
exchanging involution of $(\AAA^+,\AAA^-)$.
In this case $\AAA:=\AAA^+$
becomes an associative triple system of the second kind when
equipped with the product
$$
\langle xyz\rangle:= \Gamma(x,o^+,\tau(y),o^-,z).
$$
\item
Assume $\tau:\XX \to \XX$ is base-point preserving, and
let $a \in \cY'$ such that $o^+ \top a$ (i.e.,
$a \in \bA^-$ and $\tau(a)=-a$).
Then the Lie algebra of the group $(\GG(\tau;a),o^+)$
is the space $(\AAA^+)^{\tau^+}$ with Lie bracket
$$
[x,z]_a = 2 (\langle xaz\rangle - \langle zax\rangle).
$$
\end{enumerate}
\end{theorem}

\begin{proof}
(i), (ii):
All claims are simple applications of the functoriality of associating
an associative pair  to an
associative geometry with base pair, \cite{BeKi09},
Theorem 3.5.
For convenience, let us just spell out the computation proving the property of an
associative triple system in part ii):
\begin{eqnarray*}
\langle u\langle xyz\rangle w\rangle & =&\Gamma\Bigl(u,o^+,\tau\bigl(\Gamma(x,o^+,\tau(y),o^-,z)\bigl),
o^-,w\Bigr)
\cr
& = &\Gamma\Bigr(u,o^+,\Gamma(\tau z,\tau o^+,y,\tau o^-,\tau x),o^-,w\Bigr)
\cr
& = &\Gamma\Bigr(u,o^+,\Gamma(\tau z,o^-,y,o^+,\tau x),o^-,w\Bigr)
\cr
& = &   \Gamma\Bigr(\Gamma(u,o^+,\tau z,o^-,y),o^+,\tau x,o^-,w\Bigr)
\, \,  =  \, \,  \langle\langle uzy\rangle xw\rangle
\end{eqnarray*}
(If we had used a base point preserving \emph{automorphism} instead
of an involution, a similar calculation shows that we would get an
associative triple system of the \emph{first} kind, see Appendix A.)

(iii):  Using Lemma \ref{Liealgebra}, with $b=\tau(a)=-a$
(since the effect of $\tau$ on $\bA^-$ is multiplication by $-1$), we
get the Lie bracket
$[x,z] = \langle x (2a) z\rangle - \langle z (2a) x\rangle$.
\end{proof}

Putting the preceding two results together, we obtain an explicit description of the
groups $\GG(-\tau;b/2) \cong \UU(\tau;b,o^+,o^-)$ in terms of the associative pair
$(\bA^+,\bA^-)$:
$$
\UU(\tau;b,o^+,o^-) = \{ x \in \bA^+ \vert \,
1-xb \, \mbox{invertible},  \tau(x) = j_b(x) \}
$$
with $j_b(x)=-(1-xb)^{-1}x$, so that the condition $-\tau(x) = j_b(x) $ is equivalent to
$x+\tau(x)= \langle xb\tau(x)\rangle$.
This formulation is valid for an arbitrary associative pair with base-point
preserving involution. In practice, all known examples arise for associative
pairs corresponding to \emph{unital} associative algebras, to be discussed next.

\subsection{Base triples, unitary groups, and Cayley transform}

Next let us assume that $W$ admits a \emph{transversal triple}
$(o^+,e,o^-)$. Then
 $W = o^+ \oplus o^-$, and saying that $e$ is transversal
to $o^+$ and $o^-$ amounts saying that $e$ is the
graph of a linear isomorphism $o^+ \to o^-$. We may consider this
isomorphism as an identification, so that $e$ becomes the diagonal
$\Delta_+$ in $W = o^+ \oplus o^- = o^+ \oplus o^+$.
Then the element
$$
-e :=  M_{o^+ o^- o^- o^+} (e)
$$
becomes the antidiagonal $\Delta_-$ in  $o^+ \oplus o^+$.
In this situation, we may let the group $\Gl(2,\K)$ act  by
block-matrices on $W  = o^+ \oplus o^+$ in the usual way.
Let $G  \subset \Gl(2,\K)$ by the group generated by
$$
\begin{pmatrix} 1 & 1 \cr 0 & 1 \end{pmatrix}, \quad
 \begin{pmatrix} \lambda & 0 \cr 0 & 1 \end{pmatrix}, \quad
 \begin{pmatrix} 0 & 1 \cr 1 & 0 \end{pmatrix}
$$
with $\lambda \in \K^\times$.
The first matrix describes  left translation by $e$,
$$
L_{eo^-o^+o^-}:= 1 - P_{o^-}^e P^{o^-}_{o^+},
$$
the second multiplication by the scalar $\lambda$,
$$
\delta_{o^+ o^-}^\lambda = \lambda P^{o^+}_{o^-} + P^{o^-}_{o^+},
$$
and the third describes a map $j$ whose
effect on the associative algebra $\bA$ is inversion:
$$
j:=M_{eo^+ o^-e} = M_{o^+eeo^-} .
$$
All of these operators are (inner) automorphisms of the geometry
$(\cX,\Gamma)$.
>From (\ref{1.1})  it follows that $j$ is an automorphism of order $2$,
but this time it exchanges the points $o^+$ and $o^-$:
$$
M_{eo^+ o^-e}(o^+)=\Gamma(e,o^+,o^+, o^-,e)
= \Gamma(o^+,e,o^+,e, o^-) = o^- .
$$
Moreover, $j(e)=M_{eo^+ o^-e}(e)=\Gamma(e,o^+,e,o^-,e)=e$.

\begin{definition}
If $(o^+,e,o^-)$ is a  transversal triple, we call $\tau$ a
\begin{itemize}
\item
 \emph{unital base point preserving involution} if
$\tau(o^+)=o^+$, $\tau(o^-)=o^-$, $\tau(e)=e$,
\item
\emph{unital base point exchanging involution} if
$\tau(o^+)=o^-$, $\tau(o^+)=o^-$, $\tau(e)=e$.
\end{itemize}
\end{definition}

\nin
Note that, if $\tau$ is of one of these two types, then the dual involution
 $\tau'$ no longer preserves $e$. Indeed,
$M_{o^+ o^- o^- o^+}(e)= -e$ is the antidiagonal, which is
different from the diagonal  (if $W$ has no $2$-torsion).
Thus the r\^oles of $\tau$ and $\tau'$ are no longer completely
symmetric in the unital case.

\begin{lemma} \label{dual'}
Assume $\tau$ is a unital base point preserving
involution of $\XX$. Then $\tau$ commutes with the automorphism
$j=M_{eo^+ o^-e}$, and
$$
\tilde \tau:= j \tau = \tau j
$$
is a unital base-point exchanging involution.
Moreover, if $2$ is invertible in $\K$,
 there exists an automorphism $\rho:\cX \to \cX$ (``the
real Cayley transform'') such that
$$
\rho \circ \tau  \circ \rho^{-1} = \tau , \quad \quad
\rho \circ \tilde \tau \circ \rho^{-1} = \tau'.
$$
\end{lemma}

\begin{proof} As in the proof of
Lemma \ref{dual}, we see that
$$
\tau j \tau = \tau M_{eo^+ o^-e} \tau =
M_{eo^+ o^-e} = j,
$$
hence $j \tau$ is of order two, and it exchanges base points
and is again an involution.

The automorphism $\rho$ is constructed as follows:
let $\rho \in G$ be given by the matrix
$$
R:= \begin{pmatrix} 1 & -1 \cr 1 & 1 \end{pmatrix} =
 \begin{pmatrix} 1 & 1 \cr 0 & 1 \end{pmatrix}
 \begin{pmatrix} -2 & 0 \cr 0 & 1 \end{pmatrix}
 \begin{pmatrix} 0 & 1 \cr 1 & 0 \end{pmatrix}
\begin{pmatrix} 1 & 1 \cr 0 & 1 \end{pmatrix} .
$$
Then $\rho$ commutes with $\tau$: indeed, $\tau$ commutes with
all generators of the group $G$ mentioned above (since
 these operators are partial maps
of $\Gamma$ involving only the $\tau$-fixed elements
$o^+,o^-,e,-e$ and hence commute with $\tau$), hence $\tau$
commutes with $R$.
Since  $R$ sends the 4-tuple
$(o^-,e, o^+,-e)$ to $(e,o^+,-e,o^-)$, it follows that
$$
\rho j \rho = \rho M_{eo^+ o^-e} \rho =
M_{o^+ (-e) e o^+} =
M_{o^+ o^- o^-  o^+}
$$
(the last equality follows since
$M_{o^+ (-a) a o^+} = M_{(-a)o^+   o^+a} =
M_{o^-o^+   o^+o^-}$
is the map $x \mapsto (-a) -x + a = -x$ for all $a \in V^-$).
Together, this implies
$$
\rho \circ \tilde \tau \circ \rho^{-1} =
\rho \circ  \tau j \circ \rho^{-1} =
\tau \rho \circ j \circ \rho^{-1} =
\tau \circ M_{o^+ o^- o^-  o^+} =
\tau'.
$$
(Note that $R$ is not uniquely determined by the property from
the lemma, but the given form corresponds of course to the
well-known ``real'' version of the Cayley transform which enjoys
further nice properties.)
\end{proof}

\begin{theorem}\label{unitalMaintheorem}\label{functor'}
Assume $\tau$ is a unital base-point preserving
 involution of the Grassmannian geometry
$(\XX; o^+,e,o^-)$, and let
$\bA = C_{o^-}$ be the corresponding unital associative algebra with origin $o^+$
and $\bA^- = C_{o^+}$ the one with origin $o^-$,
let $\tau'$ the dual involution of $\tau$, $\tilde \tau = j \tau$,
$\cY:=\cX^\tau$ and
$\cY':=\cX^{\tau'}$. Let $a \in \cX$ such that $o^+ \top a$, i.e.,
$a \in \bA^-$.
\begin{enumerate}[label=\roman*\emph{)},leftmargin=*]
\item
By restriction, $\tau$ induces an involutive antiautomorphism
of $\AAA$.
This defines a functor from
the category of unital involutive associative geometries to the
category of involutive associative algebras.
\item
If  $a \in \cY'$,
then the Lie algebra of the group $\cG(\tau; a)$
is the space $\Herm(\AAA,\tau)=\AAA^{\tau}$ with Lie bracket
$[x,z]_a = 2(\langle x az\rangle - \langle zax\rangle)$. Identifying $\bA$ and $\bA^-$ via the canonical
isomorphism $j$, $a$ is identified with the element $j(a) \in \Aherm(\bA,\tau)$
and  the Lie bracket is expressed in terms of $\bA$ as
$$
[x,z]_a = 2(x az - zax) .
$$
\item
If $a \in \cY$,
then the Lie algebra of the group $\cG(\tau';a)$
is the space $\Aherm(\AAA,\tau)=\AAA^{\tau'}$ with Lie bracket
$[x,z]_a = 2(\langle xaz\rangle - \langle zax\rangle )$. With similar identifications as above,
this can be rewritten as $[x,z]_a = 2(x az - zax)$.

If, moreover, $a$ is invertible in $\bA$, then the group
$\cG(\tau';a)$ is isomorphic to the unitary group
$\UU(\bA_a,*) = \{ x \in \bA | \, x a x^* = 1 \}$
of the involutive algebra $(\bA_a,\tau)$  with product $x \cdot_a y = xay$ and involution $\tau$.
\end{enumerate}
\end{theorem}

\begin{proof}
(i)
We show that $\tau$ induces an algebra involution:
$$
\tau (xz)  =  \tau
\Gamma\bigl(x,o^+,e,o^-,z \bigr)
=\Gamma\bigl(\tau z,o^+,e,o^-,\tau x \bigr)
=(\tau z)(\tau x)
$$
Functoriality follows from \cite{BeKi09}, Theorem 3.4.

(ii)
The fixed point space of $\tau$ in $\bA$ is, by definition,
$\Herm(\bA,*)$, and by Lemma \ref{clue}, $\tau$ is an automorphism of
$U_{a \tau(a)}$. The formula from the Lie bracket follows from
Theorem \ref{functor1}. Finally, in order to relate the associative
pair to the algebra formulation, recall from \cite{BeKi09} that,
for all $a \in \bA^-$ and $x,z \in \bA^+$,
$$
\langle xay\rangle^+ = x \cdot j(a) \cdot z,
$$
where on the right hand side products are taken in the algebra $\bA$.
Since the $\K$-linear isomorphism $j:\bA^+ \to \bA^-$ commutes with
$\tau$, the formulas from the claim follow.

(iii)
The statement on the Lie algebra
 is proved in the same way as (ii), with signs changed.
Now let  $a$ be invertible. Assume first $a=1$.
Note that the condition $xx^*=1$ is equivalent to
$x=(x^*)^{-1}=j \tau(x)$, and hence $\UU(\bA,*)$ is precisely
the fixed point set of $\tilde \tau$ in $\bA$.
Its group structure is induced from $\bA^\times = U_{o^+o^-}$.
Now, the setting
$(\bA^\times,\tilde \tau)=
(U_{o^+o^-},j\tau)$ is conjugate,
via the Cayley transform $\rho$, to the setting
$(U_{e,-e},\tau')=(U_{e,\tau'(e)},\tau')$, showing that the Cayley
transform $\rho$ induces the desired isomorphism.
In these arguments, the fixed element $e \in ( \cY \cap U_{o^+o^-})$
may be replaced by any other element $a$ of this set; this simply
amounts to replacing $\bA$ by its isotope algebra $\bA_a$.
\end{proof}

\begin{theorem} \label{functor2}
Consider the following classes of objects:
\begin{description}
\item[IG] associative geometries with base triple and
base triple preserving involutions,
 \item[IA] involutive unital associative algebras.
\end{description}
There are maps $F: {\bf IG} \to {\bf IA}$ and $G: {\bf IA} \to {\bf IG}$
such that $G \circ F$ is the identity.
\end{theorem}

\begin{proof}
The map $F$ is defined by part (i) of the preceding theorem.
We define the map $G$:
given an involutive associative algebra $(\AAA,*)$,
let $\hat \cX$ be the Grassmannian of complemented right
$\bA$-submodules in $\bA^2$.
We define on $\bA^2$ the skew-Hermitian (``symplectic'') form
$$
\beta(x,y)=\overline x_1 y_2 - \overline x_2 y_1 .
$$
and consider the involution $\tau$ given by the orthocomplementation
map with respect to this form.
Let $o^+ = \bA \oplus 0$ (first factor),
$o^- = 0 \oplus \bA$ (second factor) and
$e = \Delta$ (diagonal in $\bA^2$).
Then $(o^+,e,o^-)$ is a transversal triple, preserved by $\tau$.
This defines $G$.  The associative
algebra $C_{o^-}$ associated to these data is the algebra $\bA$ we started with
(cf.\ \cite{BeKi09}, Theorem 3.5).
It remains to prove that restriction of $\tau$ to $C_{o^-} = \bA$
gives back the involution $*$ we started with.
Let $a \in \bA$ and identify it with the graph
$\{ (v,av) | \, v \in \bA \}$.
Then the graph of the adjoint operator $a^*$ is the orthogonal
complement of this graph with respect to $\beta$, whence $\tau(a)=a^*$.
\end{proof}

We have seen above that $F$ is a functor; for $G$, this is less clear  --
 cf.\ remarks in \cite{BeKi09}, Section 3.4. We will not pursue here further the discussion
 of functoriality, nor will we
state an analog of the theorem for the
non-unital case. Constructions are similar in that case, but are more
complicated (since one has to use some algebra-imbedding of
an associative pair, see \cite{BeKi09}), and practically less relevant than
the unital case.

\section{The classical torsors}

Putting together the results from the preceding two chapters, the ``projective"
description of the classical groups (Table given in the Introduction) is now straightforward:
we just have to restate Theorems \ref{conjug} and \ref{unitalMaintheorem}
for involutions given by orthocomplementation (Theorem \ref{Th:3.2}).
In the following, we list the results, first for the case of bilinear forms, then
for sesquilinear forms.

\subsection{Orthogonal and (half-) symplectic groups}
We specialize
Theorem \ref{unitalMaintheorem}
to the case $\bB=\K$, $W = \K^{2n} = \K^n \oplus \K^n$. Let
 $(o^+,e,o^-)$ be the canonical base triple
$(\K^n \oplus 0,\Delta, 0 \oplus \K^n)$
and
$\beta$ the standard symplectic form on $\K^{2n}$.
By Theorem \ref{Th:3.2}, we
have the three (restricted) involutions
$\tau$, $\tau'$, $\tilde \tau$: they are the orthocomplementation
maps with respect to the three forms given by the matrices
\begin{equation}
\Omega_n := \begin{pmatrix} 0 & 1_n \cr -1_n & 0 \end{pmatrix},
\quad
F_n:= \begin{pmatrix} 0 & 1_n \cr 1_n & 0 \end{pmatrix},
\quad I_{n,n}:=\begin{pmatrix} 1_n & 0 \cr 0 & -1_n \end{pmatrix}.
\label{sympl'}
\end{equation}
Note that $o^+$, $o^-$ and $\Delta$ are maximal isotropic for
$\beta$, hence $\tau$ is a unital base point preserving involution.
The involutive algebra
corresponding to the unital base point preserving involution $\tau$
is  $\bA = M(n,n;\K)$ with involution
$X^* = X^t$ (usual transpose).
The fixed point spaces of the three involutions
 are the classical Lagrangian varieties
corresponding to the three forms, and the tangent space
 of $\cY$ at $o^+$ is
$\Sym(n,\K)$ and the one of $\cY'$ at $o^+$ is $\Asym(n,\K)$.
Note that $\Sym(n,\K)$ is imbedded in $\cY$, and
$\Asym(n,\K)$ in $\cY'$, the subsets of elements of $\cY$ (resp.\
of $\cY'$) that are
transversal to $o^-$. Therefore the elements $a$ parametrizing
the torsors $\cG(\tau;a)$ (resp.\ $\cG(\tau';a)$) will be chosen in these
subsets. From Theorem \ref{conjug} we get:

\begin{proposition}
For $a = A \in \Sym(n,\K)$, the group $\cG(\tau;a)$ with origin $o^+$
is isomorphic to the group $\OO_n(2A,\K)$,
and for $a=A \in \Asym(n,\K)$, the group $\cG(\tau';a)$ with
origin $o^+$ is isomorphic to the group $\Sp_{n/2}(2A;\K)$.
If $2$ is invertible in $\K$, then these groups are isomorphic to
$\OO_n(A,\K)$, resp.\ $\Sp_{n/2}(A;\K)$.
\end{proposition}

Having established the link of the projective torsors $\cG(\tau;a)$ with
the affine realization of the classical torsors from the Introduction,
it is now relatively easy to \emph{classify} them
(in finite dimension over $\K=\C$ or $\R$; the
case of general base fields is much more difficult, and for general
base rings and arbitrary dimension, classification results can only
be expected under rather special assumptions).

\begin{proposition}
A complete classification of the homotopes of complex
or real orthogonal, resp.\ (half-)symplectic
groups  is given as follows:
\begin{enumerate}
\item (half-)symplectic case:
for $\K=\R,\C$, all homotopes are isomorphic to one of the groups
$\Sp_m(\Omega_r;\K)$ for $r=1,\ldots,m$ (with $n=2m$ or $n=2m+1$),
where $\Omega_r$ denotes the normal form of a skew-symmetric matrix of rank
$2r$,
\item orthogonal case:
for $\K=\C$, all homotopes are isomorphic to one of the groups
$\OO_n(1_r;\C)$ for $r=1,\ldots,n$, where $1_r$ denotes the $n \times n$-diagonal matrix of rank $r$
having first $r$ diagonal elements equal to one,

for $\K=\R$, all homotopes are isomorphic to one of the groups
$\OO_n(I_{r,s};\R)$, where $I_{r,s}$ denotes the $n \times n$-diagonal matrix of rank $r+s$
($r \leq s$, $r+s \leq n$)
having first $r$ diagonal elements equal to one and $s$ diagonal elements equal to minus one.
\end{enumerate}
\end{proposition}

\begin{proof}
One can prove the classification from a ``projective" point of view:
clearly,  if $a$ and $b$ belong to the same $\Aut(\cX,\tau)$-orbit
in $\cX$, then $\cG(\tau;a)$ and $\cG(\tau;b)$ are isomorphic,
and it is enough to consider orbits of subspaces $a\subset W$ such that $a$
and $\tau(a)$ have same dimension $n$ (otherwise $U_{a,\tau(a)}$ is empty).
Classifying such orbits is done by elementary linear algebra using
Witt's theorem:
$a$ and $b$ are conjugate iff the restriction of the
given forms to $a$, resp.\ $b$ are isomorphic.
In particular,
the totally isotropic subspaces form one orbit (the Lagrangian $\cY$).
The list of orbits then gives rise to the given list of homotopes.

Alternatively, an ``affine" version of these
arguments goes as follows: using the explicit description of the classical
groups given in the Introduction,
one notices that, e.g., $\OO_n(A;\K)$ and $\OO_n(gAg^t;\K)$
are isomorphic for all $g \in \Gl(n;\K)$; hence it suffices to
to consider the  classification of $\Gl(n;\K)$-orbits in
$\Sym(n;\K)$. This leads to the same result (note, however,  that different orbits may
give rise to isomorphic groups: e.g., $\OO_n(\lambda A;\K)$ and $\OO_n(A;\K)$ are isomorphic
whenever the scalar $\lambda$ is invertible, be it a square or not in $\K$).
Similarly for the symplectic case.
\end{proof}

\subsection{Unitary groups}

The following classification of real classical torsors associated
to involutive algebras of Hermitian type
is established in the same way as above:

\begin{proposition}
Homotopes of complex and quaternionic unitary groups are classified as follows
(see Introduction for the notation $\widetilde \bH$):

\msk \noindent
\begin{tabular}{lllll}
\noindent  & $\bA = M(n,n;\bH)$, & $\tau(X):=\overline X^t$ & &
\cr
\hline
a) & $\bA^\tau=\Herm(n,\bH)$ &
 $\UU_n(i 1_{r};\widetilde \bH)$ ($r \leq n$) &
 homotopes of $\OO^*(2n)$ &
\cr
b) & $\bA^{\tau'}=\Aherm(n,\bH)$ &
 $\UU_n(I_{r,s},\bH)$ ($r \leq s$, $r+s \leq n$) &
 homotopes of $\Sp(p,q)$ &
\end{tabular}

\ssk \noindent
\begin{tabular}{lllll}
\noindent  & $\bA = M(n,n;\C)$, & $\tau(X):=\overline X^t$ & &
\cr
\hline
a) & $\bA^\tau=\Herm(n,\C)$ &
 $\UU_n( i  I_{r,s};\C)$ ($r \leq s$, $r+s \leq n$) &
 homotopes of $\UU(p,q)$ &
\cr
b) & $\bA^{\tau'}=i\Herm(n,\C)$ &
 $\UU_n(I_{r,s};\C)$ ($r \leq s$, $r+s \leq n$)  &
 homotopes of $\UU(p,q)$ &
\end{tabular}
\end{proposition}

Over more general base fields or rings the classification
of non-degenerate torsors is essentially equivalent to the classification
of involutions of associative algebras
 -- see \cite{KMRS98} for this vast topic.

 \subsection{Hilbert Grassmannian}\label{Hilbert}

A fairly straightforward infinite dimensional generalization
of the preceding situation is the following:
$W = H \oplus H$, where $H$ is a Hilbert
space $W$ over $\bB=\C$ or $\R$, and $\beta$ corresponding to
the matrix
$$
B = \Omega_H = \begin{pmatrix} 0 & 1_H \cr -1_H & 0 \end{pmatrix}
\quad \mbox{or} \quad
B = \begin{pmatrix} 0 & 1_H \cr 1_H & 0 \end{pmatrix} .
$$
In this case we may work with the Grassmannian of all \emph{closed}
subspaces of $W$, and it easily seen that all arguments from the
proof of Theorem \ref{Th:3.2} go through, showing that the
orthocomplementation map of $\beta$ defines an involution of this
geometry. We get infinite dimensional analogs of the classical
groups, imbedded, together with their homotopes, in Hilbert-Lagrangian
manifolds. Variants of these constructions can be applied to
\emph{restricted Grassmannians} and \emph{restricted unitary groups}
in the sense of \cite{PS86}.

\section{Semitorsors}


In this chapter we extend our theory from \emph{restricted} involutions to
``globally defined'' involutions. Roughly speaking, the restricted product map $\Gamma$ and
the corresponding restricted involutions deal with connected geometries (the
``restricted'' theory developed so far is, in spite of its algebraic flavor, analoguous to the
correspondence between Lie algebras and \emph{connected} Lie groups), whereas
the global product map $\Gamma$ and its global involutions rather correspond
to replacing connected Lie groups by \emph{algebraic groups}.

\subsection{Semigroup completion of general linear groups}

Let $W$ be a right $\bB$-module and $\cX$ its Grassmannian.
In \cite{BeKi09} we have shown that the torsors $U_{ab} \subset \cX$
admit a ``semitorsor completion'': the ternary law $(xyz)$ from
$U_{ab}$ extends to the whole of $\cX$, given by the formula
\begin{equation}
\Gamma(x,a,y,b,z):=
\Bigsetof{\omega \in W}
{\begin{array}{c}
\exists \xi \in x,
\exists \alpha \in a,
\exists \eta \in y,
\exists \beta \in b,
\exists \zeta \in z : \\
\omega = \zeta + \alpha
= \zeta + \eta + \xi
= \xi + \beta
\end{array}}\, .
\end{equation}
This formula defines a quintary ``product map'' $\Gamma:\cX^5 \to \cX$
having the following remarkable properties:
for any fixed pair $(a,b)$, the partial map
$(xyz):=\Gamma(x,a,y,b,z)$ satisfies the \emph{para-associative law}
\begin{equation}
(xy(zuv))= (x(uzy)v)=((xyz)uv),
\end{equation}
and it is invariant under the Klein 4-group acting on $(x,a,b,z)$:
\begin{equation}
\Gamma(x,a,y,b,z)=\Gamma(a,x,y,z,b)=\Gamma(z,b,y,a,x).
\end{equation}
We say that, for $a,b$ fixed, $\cX$ with $(xyz)=\Gamma(x,a,y,b,z)$ is
a \emph{semitorsor}, denoted by $\cX_{ab}$ (for fixed $y$, it is in
particular a semigroup), and $\cX_{ba}$ is its \emph{opposite semitorsor}.
For simplicity, we are not going to consider here the globally defined dilation maps
$\Pi_r$ from \cite{BeKi09}; in other words, for the moment we look at $\cX$
as an associative geometry defined over $\Z$ (in fact, one has to be very careful
with the globally defined maps $\Pi_r$ as soon as $r$ or $1-r$ is not invertible;
in order to keep this work in reasonable bounds we postpone a more detailed
discussion of these problems).

\begin{definition}
An \emph{involution} of the Grassmannian geometry
$\cX=\Gras(W)$ is a bijection $\tau:\cX \to \cX$ of order $2$ such that,
for all $x,a,y,b,z \in \cX$, without any restriction by transversality conditions,
$$
\tau (\Gamma(x,a,y,b,z)=\Gamma(\tau(z),\tau(a),\tau(y),\tau(b),\tau(x)) \, .
$$
\end{definition}
The following lemma is proved exactly as Lemma \ref{clue}:

\begin{lemma}\label{clue'} Let $\tau:\XX \to \XX$ be an
involution
of the Grassmannian geometry $\XX = \Gras_\BB(W)$, let
$\YY=\XX^\tau$ and $a \in \cX$.
Then
$\tau$ induces a semitorsor-automorphism of
$\cX_{a,\tau(a)}$. In particular, the fixed point set
$\cY$
is a subsemitorsor of $\cX_{a,\tau(a)}$.
If $a  \in \cY$, then the semitorsor $\cX_{a,\tau(a)} \cap \cY$ is abelian.
\end{lemma}

Since the globally defined product map $\Gamma$ encodes the lattice structure of $\Gras(W)$,
an involution $\tau$ induces an \emph{involution of the underlying lattice} (\cite{BeKi09},
Theorem 2.4 and Section 3.1).
Hence the condition that $\tau$ is a lattice involution
is \emph{necessary}, and thus orthocomplementation maps are the natural candidates.
Our tool for proving that they indeed define involutions
is the notion of \emph{generalized projection}, which might be of independent
interest for the theory of linear relations.

\subsection{Generalized projections}
Linear operators $f \in \End_\bB(W)$ are generalized by
 \emph{linear relations in $W$}, i.e., submodules
$F \subset W \oplus W$. Following standard terminology
(see, e.g., \cite{Ner96}, \cite{Cr98}), \emph{domain}, \emph{image},
\emph{kernel} and \emph{indefiniteness} of $F$ are the subspaces
defined by
$$
\dom F := \pr_1 F, \quad \im F := \pr_2 F, \quad
\ker F := F \cap (W \times 0), \quad
\indef F := F \cap (0 \times W)
$$
with $\pr_i:F \to W$ the two projections.
For any $a,b \in \cX$,
define the linear relation $P_x^a \subset W \oplus W$,
called a \emph{generalized projection}, by
\begin{equation}
P_x^a := \big\{ (\zeta,\omega) | \, \omega \in x, \, \omega - \zeta \in a \big\}.
\end{equation}
Note that
$$
 \im P_x^a= x, \quad
\ker P_x^a= a, \quad
\indef P_x^a= a \land x, \quad
\dom P_x^a = x \lor a,
$$
and that, if $a \top x$, then $P_x^a$ is the graph of the projection
denoted previously by $P_x^a$, so
there should be no confusion with preceding notation.
We denote the \emph{space of generalized projections} by
$$
\cP: =  \{ P_x^a | \, x,a \in \cX \} \subset \Gras(W \oplus W).
$$
The map
$$
\cX \times \cX \to \cP, \quad (a,x) \mapsto P_x^a
$$
is a bijection with inverse $P \mapsto (\ker P,\im P)$.
Transversal pairs $(x,a)$ correspond to ``true'' operators (single
valued and everywhere defined).
%
%

\begin{lemma}\label{Idem}
The linear relation
$P_x^a$ is idempotent:
$P_x^a \circ P_x^a =P_x^a$.
\end{lemma}

\begin{proof} By definition of composition,
$$
P_x^a \circ P_x^a = \{ (u,w) | \exists v \in W:
v \in x, u-v \in a, w \in x, v-w \in a \}.
$$
Since $w \in x$ and $w-u=(w-v)+(v-u) \in a$, we have
$P_x^a  \circ P_x^a \subset P_x^a$.
For the other inclusion, let $(u',w') \in P_x^a$,
so $w' \in x$, $w'-u'  \in a$.
Let $u:=u'$, $w:= v:=w'$; then
$v,w\in x$ and $u-v = u' - w' \in a$, $v-w =0 \in a$,
whence $(u',w') \in P_x^a \circ P_x^a$.
\end{proof}

\begin{lemma} \label{conjugation}
The set $\cP$ of generalized projections is stable under
``conjugation'' by linear relations in the following sense:
 for all linear
relations $F \subset W \oplus W$ and all $c,z \in \cX$, we have
$$
F \circ P^c_z \circ F^{-1} = P^{F(c)}_{F(z)}.
$$
\end{lemma}

\begin{proof}
By definition of composition and inverse,
\begin{eqnarray*}
F \circ P^c_z \circ F^{-1} &=&
\{ (\alpha,\delta) | \, \exists \beta, \gamma \in W:
\, (\alpha,\beta) \in F^{-1}, (\beta,\gamma) \in P^c_z,
(\gamma,\delta) \in F \}
\cr
&=& \{ (\alpha,\delta) | \, \exists \beta \in W, \gamma \in z:
\, (\beta,\alpha) \in F,
(\gamma,\delta) \in F,
 \gamma - \beta \in c \}
\end{eqnarray*}
These conditions imply that $\delta \in Fz$ and
$(\beta,\alpha)-(\gamma,\delta) \in F$;
since $(\beta - \gamma) \in c$, this implies also
$(\alpha - \delta) \in Fc$. It follows that
$(\alpha,\delta) \in P^{F(c)}_{F(z)}$.

Conversely, let $(\alpha,\delta) \in P^{F(c)}_{F(z)}$,
i.e., $\delta \in F(z)$, $\alpha - \delta \in F(c)$,
so there exists $\gamma \in z$ with $(\gamma,\delta) \in F$ and
$\eta \in c$ with $(\eta,\alpha - \delta) \in F$.
Let $\beta := \gamma - \eta$, so
$\gamma - \beta \in c$ and
$$
(\beta, \alpha) = (\gamma,\delta) - (\eta,\delta - \alpha) \in F ,
$$
whence $(\alpha,\delta) \in F \circ P^c_z \circ F^{-1}$.
%
%
\end{proof}

For the next statements, recall (\cite{Ar61}, \cite{Cr98}) the following general definitions
concerning linear relations.
For a linear relation $F \subset W \oplus W$ and $z \in \cX$,
the \emph{image of $z$ under $F$} is
$$
Fz := F(z):=
 \{ \delta \in W | \, \exists \gamma \in z: (\gamma,\delta) \in F \}
= {\pr}_2 ({\pr}_1)^{-1}(z),
$$
and  the
\emph{difference of linear relations} $F,G \subset \Gras(W \oplus W)$, is
$$
F - G := \{ (\xi,\omega) | \exists \alpha, \beta \in W:
(\xi,\alpha) \in F, (\xi,\beta) \in G, \omega=\alpha - \beta \}
$$
Remark: This difference 
can also be
written in our language in terms of the associative geometry
$(\Gras(W \oplus W),\hat \Gamma)$, with its usual base points $o^+,o^-$, as
$$
F - G := \hat \Gamma (F,o^-,G,o^-,o^+),
$$
the difference of $F$ and $G$ in the linear space $(C_{o^-},o^+)$.

\begin{lemma} \label{opp} For all $a,x \in \cX$,
$1 - P^a_x = P^x_a$.
\end{lemma}

\begin{proof}
 $\omega = u - \omega'$ with $\omega' \in x$, $\omega' - u \in a$
is equivalent to $\omega \in a$ with $\omega - u \in x$.
\end{proof}

\begin{theorem} \label{Newgamma}
Let $\Gamma$ be the multiplication map of the
 Grassmann geometry $\cX$.
\begin{enumerate}
\item
For all $(x,a,y,b,z) \in \cX^5$,
$$
\Gamma(x,a,y,b,z)=
(1 - P^x_a P^b_y) (z) = (P^a_x - P^z_b )(y) .
$$
In other words, the left multiplication operator $L_{xayb}$ in the geometry
$(\cX,\Gamma)$ is
induced by the linear relation $1 - P^x_a P^b_y$,
and the middle multiplication operator $M_{xabz}$
 is induced by the linear relation
$P^a_x - P^z_b$. Thus we can (and will) define,
extending the operator notation from Chapter 1, the linear relations
$$
L_{xayb}:= 1 - P^x_a P^b_y , \quad
M_{xabz}:= P^a_x - P^z_b .
$$
\item For all $(x,a,z) \in \cX^3$,
$$
P^a_x(z) = L_{xaax}(z)= \Gamma(x,a,a,x,z)= x \land (a \lor z) .
$$
\item
For all $a,b,x,y \in \cX$, using Notation from part (1),
$$
L_{xayb}^{-1}(z)=L_{yaxb}(z), \quad
M_{xabz}^{-1}(y)=M_{zabx}(y).
$$
In particular
$$
(P^a_x)\inv(z)=L_{xaax}\inv(x)=L_{aaxx}(z)=\Gamma(a,a,x,x,z)=
a \lor (x \land z).
$$
\end{enumerate}
\end{theorem}

\begin{proof} (1)
Note that, under certain transversality conditions ensuring that
the linear relations in question are indeed graphs of linear
operators, the claim has already been proved in \cite{BeKi09}.
Let us prove it now in the general situation.
\begin{eqnarray*}
P^x_a \circ P^b_y &=&
\{ (\zeta,\omega) | \exists \eta \in y :
\zeta - \eta \in b, \omega - \eta \in x, \omega \in a \},
\cr
1 - P^x_a P^b_y &=&
\{ (\zeta,\omega') | \exists \omega \in W: (\zeta,\omega) \in  P^x_a P^b_y,
\omega' = \zeta - \omega \}
\cr
&=&
 \{ (\zeta,\omega') |
\exists \omega \in W, \exists \eta \in y :
\zeta - \eta \in b, \omega - \eta \in x, \omega \in a,
\omega' = \zeta - \omega \}
\end{eqnarray*}
whence
\begin{eqnarray*}
(1 - P^x_a P^b_y)(z) &=&  \{ \omega' \in W |
\exists \zeta \in z, \exists \alpha \in a, \exists \eta \in y :
\zeta - \eta \in b, \alpha - \eta \in x,
\omega'= \zeta - \alpha \}
\cr
&=& \{ \omega' \in W |
 \exists \alpha \in a, \exists \eta \in y :
\omega' + \alpha - \eta \in b, \alpha - \eta \in x,
\omega' + \alpha \in z   \}
\end{eqnarray*}
According to the ``$(a,y)$-description'' from \cite{BeKi09}, this
set is indeed equal to $\Gamma(x,a,y,b,z)$.
Similarly,
\begin{eqnarray*}
P^a_x - P^z_b &=&
\{ (\eta,\omega) | \exists u,v : (\eta,u) \in P^a_x, (\eta,v) \in P^z_b,
\omega = u -v \}
\cr
&=& \{ (\eta,\omega) | \exists u \in x, \exists v \in b :
u - \eta \in a, v - \eta \in z,
\omega = u -v \}
\end{eqnarray*}
so that
\begin{eqnarray*}
(P^a_x - P^z_b)(y)
&=& \{ \omega | \exists u \in x, \exists v \in b, \exists \eta \in y :
u - \eta \in a, v - \eta \in z,
\omega = u -v \}
\cr
&=& \{ \omega |  \exists \beta \in b, \exists \eta \in y :
\beta+ \omega - \eta \in a, \beta - \eta \in z,
\beta+ \omega \in x \}
\end{eqnarray*}
Again, by the $(y,b)$-description, this equals  $\Gamma(x,a,y,b,z)$.

(2) Using Lemmas \ref{Idem} and \ref{opp},
$\Gamma(x,a,a,x,z)=(1 - P^x_a P^x_a)(z) = (1 -P^x_a)(z) =  P^a_x(z)$,
proving the first equality. The second equality is proved in \cite{BeKi09}, Theorem 2.4 (vi).

(3)
This is a restatement of Theorem 2.5 from \cite{BeKi09}.
\end{proof}

%

\subsection{Orthocomplementation maps and adjoints}

\begin{theorem} \label{semitorsor}
Assume
$\beta$ is a non-degenerate Hermitian or skew-Hermitian  form on the right $\bB$-module
$W$, and let $\cX=\Gras(W)$.
\begin{enumerate}
\item
For all $x,a,y,b,z \in \cX$, we have the inclusion
\begin{equation*}
\Gamma(x^\perp,b^\perp,y^\perp,a^\perp,z^\perp) \subset
\big( \Gamma(x,a,y,b,z) \big)^\perp .
\label{inv'}
\end{equation*}
\item
Assume that $\bB$ is a skew-field and $W = \bB^n$.
Then equality holds in (\ref{inv'}), and
the orthocomplementation map
 is an involution of $\cX$.
  \end{enumerate}
\end{theorem}

\begin{proof} We define the \emph{adjoint relation} of a linear relation $F \subset W \oplus W$ by
$$
F^* := \{ (v',w') | \forall (v,w) \in F :
\beta(v',w)=\beta(w',v) \} \subset W \oplus W.
$$
This is the orthocomplement of $F$ with respect to the
``symplectic form'' $\Omega$ on $V \oplus V$ associated to $ \beta$,
$$
\Omega((u,v),(u',v')) = \beta(u,v') - \beta(v,u')
$$
(see  \cite{Ar61}, \cite{Cr98}, Ch.\ III). Note that $*$ and inversion commute.

\begin{lemma} \label{innocent}
For all $F \in
\Gras(W \oplus W)$ and all $z \in \Gras(W)$, we have
 $$
 (Fz)^\perp \supset  (F^*)^{-1} z^\perp .
 $$
\end{lemma}

\begin{proof}
%
Assume $v \in (F^*)\inv z^\top$.
This means there is $u \in W$ with $(v,u) \in F^*$ and $\beta(u,z)=0$.
Hence, for all $(\zeta,\zeta') \in F$ with $\zeta \in z$, we have
$0=\beta(u,\zeta)=\beta(v,\zeta')$.
Thus, whenever $\zeta' \in F(z)$, we have $\beta(v,\zeta')=0$, that is,
$v \in (Fz)^\perp$.
\end{proof}


\begin{lemma}
For any non-degenerate form $\beta$, we have
$$
(P^a_x)^* \supset  P^{x^\perp}_{a^\perp}
$$
and
$$
(P^x_a P^b_y)^*  \supset  (P^b_y)^* (P^x_a)^* \supset P_{b^\perp}^{y^\perp} P_{x^\perp}^{a^\perp}
$$
with equality in all cases under the assumptions of part (2) of the theorem.
\end{lemma}

\begin{proof}
By definition of the adjoint,
\begin{eqnarray*}
(P^a_x)^* &=&
\{ (v',w') | \forall (v,w) \in  P^a_x:
\beta(v',w)=\beta(w',v) \}
\cr
&=&
\{ (v',w') |  w \in x, v-w \in a \Rightarrow
\beta(v',w)=\beta(w',v) \}
\end{eqnarray*}
Now assume $(v',w') \in P^{x^\perp}_{a^\perp}$,
that is, $w' \perp a$, $w'-v' \perp x$.
Then, for all $w \in x$ and $v$ with $v-w \in a$:
$$
\beta(v',w)= \beta(v'-w',w) + \beta(w',w)= \beta(w',w)=
\beta(w',w-v)+\beta(w',v)=\beta(w',v),
$$
whence $(v',w') \in (P^a_x)^* $,
proving the first inclusion.
The inclusion
\begin{equation} \label{Ar}
 (G \circ F)^* \supset F^* \circ G^*
\end{equation}
holds for all linear relations $F,G$,
see \cite{Ar61}, Lemma 3.5, where it is also proved that equality always holds in
the case of finite dimension over a field.

In order to finish the proof, it only remains to show that $(P^a_x)^* = P^{x^\perp}_{a^\perp}$
under the assumptions of part (2) of the theorem. In view of the inclusion just proved,
it es enough to prove that both subspaces in question have the same dimension over $\bB$.
First of all, for every linear relation
$F$, since $\pr_2|_F$ induces an exact sequence
$0 \to \ker F \to F \to \im F \to 0$,
$$
\dim F = \dim (\ker F) + \dim (\im F)
$$
hence
$$
\dim P^a_x = \dim (a) + \dim (x), \quad
\dim P^{x^\perp}_{a^\perp} = \dim (a^\perp) + \dim (x^\perp) .
$$
Since $(P^a_x)^*$ is the orthogonal complement of $P^a_x$
with respect to a non-degenerate form on $W \oplus W$,
$$
\dim (P^a_x)^* = \dim (W \oplus W) - \dim P^a_x  =
\dim W - \dim x + \dim W - \dim a =
\dim P^{x^\perp}_{a^\perp},
$$
proving the claim.
%
\end{proof}

\begin{lemma} \label{addition}
For all linear linear relations $F \subset W \oplus W$:
$$
(1 + F)^* = 1 + F^*, \quad (1-F)^* = 1 - F^* .
$$
\end{lemma}

\begin{proof}
One checks easily that the following two linear isomorphisms of $W \oplus W$
$$
A(v,w)= (v,v+w), \quad D(v,w)=(v,v-w)
$$
preserve the form $\Omega$, and hence they are compatible with orthocomplements with
respect to $\Omega$. The claim follows by observing that
$1+F=A .F$ and $1-F=D.F$ (where the dot denotes the canonical push-forward
action of $\Gl(W \oplus W)$ on linear subspaces).
%
%
\end{proof}
From the preceding two lemmas it follows that
\begin{eqnarray} \label{left!}
(L_{xayb})^* &=&  (1 - P^x_a P^b_y)^* =
1 - (P^x_a P^b_y)^*
\cr
& \supset &  1 - (P^b_y)^*(P^x_a)^*
\cr
& \supset &
1 - P^{a^\perp}_{x^\perp} P^{y^\perp}_{b^\perp}  =
L_{y^\perp b^\perp x^\perp a^\perp} .
\end{eqnarray}
with equality under the assumptions of part (2). Now we prove part (1):
\begin{eqnarray*}
\Gamma(a,x,b,y,z)^\top = (L_{xayb}z)^\perp & = &  ((1 - P^x_a P^b_y)z)^\perp
\cr
& \supset &  ((1-P^x_a P^b_y)^*)\inv z^\perp
\cr
&\supset &
(L_{y^\perp b^\perp x^\perp a^\perp})\inv z^\perp
\cr
&=&
L_{x^\perp b^\perp y^\perp a^\perp}  z^\perp  = \Gamma(x^\perp , b^\perp , y^\perp , a^\perp ,  z^\perp)
\end{eqnarray*}
Next assume that $\bB$ is a skew-field and $W = \bB^n$. Then the second inclusion
becomes an equality, but we do not know whether the inclusion from Lemma \ref{innocent}
always becomes an equality. Therefore we will invoke Lemma \ref{conjugation}:
Choose an auxiliary element
$c \in \cX$. Then, using the fact that equality holds in (\ref{Ar}), along with Lemma \ref{conjugation}
and $(L_{xayb})^*=L_{y^\perp b^\perp x^\perp a^\perp}$, we get, on the one hand,
\begin{eqnarray*}
(L_{xayb}^{-1})^* (P_z^c)^* (L_{xayb})^* &=&(L_{xayb} P_z^c L_{xayb}^{-1})^*
\cr
&=& (P^{( \Gamma(x,a,y,b,c) )}_{( \Gamma(x,a,y,b,z))})^*
\cr
&=&
P^{( \Gamma(x,a,y,b,z) )^\perp}_{( \Gamma(x,a,c,b,z) )^\perp}
\end{eqnarray*}
and on the other hand,
\begin{eqnarray*}
(L_{xayb}^{-1})^* (P_z^c)^* (L_{xayb})^*
 &=&
(L_{xayb}^*)^{-1} (P_z^c)^* (L_{xayb})^* \cr
 &=&
L_{y^\perp b^\perp x^\perp a^\perp}^{-1}
 P_{c^\perp}^{z^\perp}
L_{y^\perp b^\perp x^\perp a^\perp}
\cr
&=&
L_{x^\perp b^\perp y^\perp a^\perp}
 P_{c^\perp}^{z^\perp}
L_{x^\perp b^\perp y^\perp a^\perp}^{-1}
\cr
&=&
P^{\Gamma(x^\perp,b^\perp,y^\perp,a^\perp,z^\perp)}_{
\Gamma(x^\perp,b^\perp,y^\perp,a^\perp,c^\perp)} \,  .
\end{eqnarray*}
Comparing images and kernels of these projections yields
the desired equality.
%
\end{proof}
With Lemma \ref{clue'}, the theorem implies

\begin{corollary}
All  classical groups over fields or skew-fields, and all of their
homotopes $\cG(\tau;a)$, admit a canonical semigroup
completion $\cX_{a,\tau a} \cap \cY$ (which is a compactification if
$\K=\R$ or $\C$).
\end{corollary}

\begin{remark}
The \emph{classification} of classical semitorsors $\cX_{a,\tau a}$ is only slightly
 more complicated than the one of the torsors $\cG(\tau_a)$ from Chapter 4:  it suffices to classify
all orbits of $\Aut(\cX)$ in $\cX \times \cX$, resp.\ all
$\Aut(\cY)$-orbits in $\cX$.  However,
 the internal structure of the semitorsors may be very complicated!
In other words, the classification of \emph{semigroups} is much more difficult
than the one of \emph{semitorsors} (a semitorsor contains many semigroups).
\end{remark}

Finally, we conjecture that part (2) of Theorem \ref{semitorsor} still holds in the context of
Hilbert Lagrangians (see Subsection \ref{Hilbert}),
providing semitorsor-completions of infinite-dimensional classical groups and their
homotopes. This conjecture is supported by the fact that the orthocomplementation map
of a Hilbert Grassmannian is a lattice involution. However, our proof uses finite-dimensionality
at several places, and thus does not generalize directly to this setting.

\section{Towards an axiomatic theory}

In a way similar to the intrinsic-axiomatic description of associative geometries from
Chapter 3 of Part I, we would like to describe axiomatically
the Lagrangian geometries $\cY$ with their torsor
and semitorsor structures -- so far they are  only defined
by \emph{construction} and not by intrinsic properties.
What are these properties?
Certainly, on the one hand, the various group and torsor structures seem to be
the most salient feature. But, on the other hand, there is an underlying
``projective" structure playing an important r\^ole -- indeed, Lagrangian geometries
are special instances of \emph{generalized projective geometries} as defined in
\cite{Be02}. This is most obvious on the ``infinitesimal" level of the corresponding
algebraic structures: besides the Lie algebra structures $[x,y]_a$ which
correspond to the groups and torsors, there are also \emph{Jordan algebra}
structures $x \bullet_a y = (xay + yax)/2$, and Jordan structures correspond precisely
to generalized projective geometries. There are purely algebraic concepts
combining these two structures (``Jordan-Lie" and "Lie-Jordan algebras";
cf.\ \cite{E84}, \cite{Be08c}), and the Lagrangian geometries considered
here should be their geometric counterparts.
In particular, the infinite dimensional
Hilbert Lagrangian geometry then is the geometric analog of the Jordan-Lie
algebra of observables in Quantum Mechanics -- see \cite{Be08c} for a
discussion of some motivation coming from physics.

\section*{Appendix A: Associative pairs and their involutions}

Recall (e.g., from \cite{BeKi09}, Appendix B, or \cite{Lo75}) that an
 \emph{associative pair (over $\bK$)} is a pair $(\bA^+,\bA^-)$ of
$\bK$-modules together with two trilinear maps
\[
\langle \cdot,\cdot,\cdot \rangle^\pm :\bA^\pm \times \bA^\mp \times \bA^\mp \to
\bA^\pm
\]
such that
\[
\langle xy \langle zuv\rangle^\pm \rangle^\pm = \langle\langle xyz\rangle^\pm uv\rangle^\pm=
\langle x\langle uzy\rangle^\mp v\rangle^\pm .
\]

\begin{definition}
A \emph{type preserving involution} of $(\bA^+,\bA^-)$ is a pair
 $(\tau^+:\bA^+ \to \bA^+,\tau^-:\bA^- \to \bA^-)$ of $\K$-linear mappings such
that $\tau^\pm$ are of order $2$ and
$$
\tau^\pm \langle uvw\rangle^\pm = \langle \tau^\pm w,\tau^\mp v, \tau^\pm u\rangle^\pm.
$$
A \emph{type exchanging involution} of $(\bA^+,\bA^-)$ is a pair
 $(\tau^+:\bA^+ \to \bA^-,\tau^-:\bA^- \to \bA^+)$ of $\K$-linear mappings such
that $\tau^+$ is the inverse of $\tau^-$ and
$$
\tau^\pm \langle uvw\rangle^\pm = \langle \tau^\mp w,\tau^\pm v, \tau^\mp u\rangle^\pm.
$$
In other words, a type preserving involution is an isomorphism onto the opposite
pair of $(\bA^+,\bA^-)$, and a type exchanging involution is an
isomorphism onto the dual of the opposite pair, where
the \emph{opposite pair} is obtained by reversing orders in
products, and the \emph{dual pair} is obtained by exchanging the
r\^oles of $\bA^+$ and $\bA^-$.
\end{definition}

Clearly, for any involution $\tau=(\tau^+,\tau^-)$, the pair
$\tau':=(-\tau^+,-\tau^-)$ is again an involution (type preserving,
resp.\ exchanging iff so is $\tau$); we call it the \emph{dual
involution}. For a type preserving involution, the pairs of
$1$-eigenspaces or of $-1$-eigenspaces in general
do not form associative pairs (but they are \emph{Jordan pairs}, see
\cite{Lo75}). For type exchanging involutions, there is an
equivalent description in terms of  \emph{triple systems\/}:
recall that
an \emph{associative triple system of the second kind} is a $\K$-module $\bA$ together
with
a trilinear map $\bA^3 \to \bA$, $(x,y,z) \mapsto \langle xyz\rangle$ satisfying
the preceding identity obtained by omitting superscripts
(see  \cite{Lo72}), and
an \emph{associative triple system of the first kind}, or
\emph{ternary ring}, is a $\K$-module $\bA$ together
with
a trilinear map $\bA^3 \to \bA$, $(x,y,z) \mapsto \langle xyz\rangle$ satisfying
the identity
\[
\langle xy \langle zuv\rangle\rangle  \, = \,  \langle\langle xyz\rangle uv\rangle
\, = \, \langle x\langle yzu\rangle v\rangle
\]
(see \cite{Li71}).
It is easily checked that, if $(\tau^+,\tau^-)$ is a type exchanging involution,
the space $\AAA:=\AAA^+$ with
$$
\langle x,y,z\rangle:= \langle x,\tau^+ y,z\rangle^+
$$
becomes an associative triple system of the second kind. Conversely,
from an associative triple system of the second kind we may reconstruct
an associative pair with  type exchanging involution:
$\bA^+:=\bA =:\bA^-$,
$ \langle x,\tau^+ y,z\rangle^\pm := \langle x,y,z\rangle$, $\tau^\pm$ given by the identity map
of $\bA^\pm \to \bA^\mp$.

In the same way, \emph{automorphisms of order two} from $(\bA^+,\bA^-)$
onto the opposite pair $(\bA^-,\bA^+)$ correspond to associative
triple systems of the first kind.

\ssk \nin
{\bf Examples.}
1. Every associative algebra $\bA$ with
$\langle xyz\rangle =xyz$ is an associative triple system of the first kind.
It is equivalent to the associative pair $(\bA,\bA)$ with the
exchange automorphism (which is not an involution, in our terminology).

\ssk
2. The space of rectangular matrices $M(p,q;\K)$ with
$$
\langle XYZ\rangle = XY^t Z
$$
forms  an associative triple system of the second kind.
It is equivalent to the associative pair
$(\bA^+,\bA^-)=(M(p,q;\K),M(q,p;\K))$
with type exchanging involution $X \mapsto X^t$.

\ssk
3. For any involutive algebra $(\bA,*)$, the map
$(x,y) \mapsto (x^*,y^*)$ is a type preserving involution of the
associative pair $(\bA,\bA)$.

\ssk \nin
{\bf Remark.} We do not have
an example of an associative pair with a type preserving involution
which is \emph{not} obtained via Example 3 above.
In finite dimension over a field the existence of such  examples seems
rather unlikely, but there might exist
 infinite dimensional examples which are ``very close'', but not isomorphic,
to pairs of the type  $(\bA,\bA)$, and admit a type-preserving involution.

%

%
\end{document}